\documentclass[10pt,twoside,leqno]{article}
\usepackage{amssymb,amsmath,amsthm,enumerate,shuffle,mathdots,stmaryrd,hyperref}
\usepackage[dvipdfmx]{graphicx}
\numberwithin{equation}{section}
\usepackage{ytableau}

\theoremstyle{plain}
\newtheorem{thm}{Theorem}[section]
\newtheorem{prop}[thm]{Proposition}
\newtheorem{lemma}[thm]{Lemma}
\newtheorem{cor}[thm]{Corollary}

\theoremstyle{definition}

\newtheorem{remark}[thm]{Remark}
\newtheorem{example}[thm]{Example}

\newenvironment{MSC}{%
\smallbreak
\noindent \textbf{2010\ Mathematics Subject Classification\,:}}

\newenvironment{keywords}{%
\noindent\textbf{Key words and phrases\,:}\itshape}

\newcommand{\hg}[4]{{\,{}_2F_1\Big(\begin{array}{c} #1, #2\\ #3 \end{array} \Bigl\vert \,#4 \Big)}}

\newcommand{\adots}{\raisebox{-2pt}{$\iddots$}} 
\newcommand{\svdots}{\raisebox{-2pt}{$\vdots$}} 
 
\newcommand{\vsmall}{\rotatebox[origin=c]{-90}{$<$}}
\definecolor{gray}{RGB}{230,230,230}  
\definecolor{darkgray}{RGB}{130,130,130}  
\newcommand{\Z}{{\mathbb Z}}
\newcommand{\N}{{\mathbb N}}
\newcommand{\Q}{{\mathbb Q}}
\newcommand{\R}{{\mathbb R}}
\newcommand{\kk}{{\bf k}}
\DeclareMathOperator{\SSYT}{\operatorname{SSYT}}
\DeclareMathOperator{\Cor}{\operatorname{C}}

\title{Checkerboard style Schur multiple zeta values \\ and odd single zeta values}
\author{Henrik Bachmann and 
Yoshinori Yamasaki\thanks{Partially supported by JSPS Grant-in-Aid for Scientific Research (C) No. 15K04785.}} 

\date{\today}

\begin{document}

\setlength{\baselineskip}{13pt}
\maketitle

 \begin{abstract} 
We give explicit formulas for the recently introduced Schur multiple zeta values, which generalize multiple zeta(-star) values and which assign to a Young tableaux a real number. In this note we consider Young tableaux of various shapes, filled with alternating entries like a Checkerboard. In particular we obtain new sum representation for odd single zeta values in terms of these Schur multiple zeta values. As a special case we show that some Schur multiple zeta values of Checkerboard style, filled with $1$ and $3$, are given by determinants of matrices with odd single zeta values as entries. \\

 \begin{MSC}
  11M41, 
  05E05, 
  33C05. 
 \end{MSC} 
 \end{abstract}



\section{Introduction}

The purpose of this paper is to give explicit formulas for special values of a certain class of Schur multiple zeta values. Schur multiple zeta values, which were recently introduced in \cite{NakasujiPhuksuwanYamasaki}, generalize the classical multiple zeta(-star) values. For $k_1,\dots,k_{r-1} \geq 1, k_r \geq 2$ the multiple zeta value and the multiple zeta-star value are defined by
\begin{equation} \label{eq:mzv}
 \zeta(k_1,\ldots,k_r)=\sum_{0<m_1<\cdots<m_r} \frac{1}{m_1^{k_1}\cdots m_r^{k_r}} \,,\quad
 \zeta^\star(k_1,\ldots,k_r)=\sum_{0<m_1\leq\cdots \leq m_r} \frac{1}{m_1^{k_1}\cdots m_r^{k_r}} \,.
\end{equation}
There are various results on special values of these real numbers. Denote by $\{k_1,\dots,k_r\}^n$ the index set which consists of $n$ repetitions of $k_1,\dots,k_r$. Then it was first proven in \cite{BBB} that for all $n\geq 1$
\begin{equation}\label{eq:31formula}
\zeta(\{1,3\}^n) = \frac{2\pi^{4n}}{(4n+2)!} = \frac{1}{4^n} \zeta(\{4\}^n)\,.
\end{equation}
In this note we are interested in analogues formulas for Schur multiple zeta values. Schur multiple zeta values generalize multiple zeta values by replacing an index set $(k_1,\dots,k_r)$ by a Young tableau.  For example for numbers $a,b,d\geq 1,\,c,e,f \geq 2$ the following sum is an example for a Schur multiple zeta value
\begin{equation} \label{eq:3stair}
\zeta\left(\ {\footnotesize \ytableausetup{centertableaux, boxsize=1.2em}
	\begin{ytableau}
	a & b & c \\
	d & e\\
	f
	\end{ytableau}}\ \right) = \sum_{
	\arraycolsep=1.4pt\def\arraystretch{0.8}
	\begin{array}{cccc}
m_a & \leq m_b & \leq m_c \\
\vsmall & \,\,\,\,\,\vsmall \\
 m_d & \leq m_e\, & \, \\
\vsmall & \\
 m_f & \, & \, 
	\end{array} } \frac{1}{m_a^a \cdot m_b^b \cdot m_c^c \cdot m_d^d \cdot m_e^e \cdot m_f^f}  \,,
\end{equation}
where we always assume that the $m_\ast$ appearing in the summation are positive integers.
We will not just consider full Young tableaux but also their skew-type, i.e. where one allows to subtract another Young tableau from the upper left corner. For example the following sum is an example of a Schur multiple zeta value for a skew Young tableau
\begin{equation*} \label{eq:skew}
\zeta\left(\ {\footnotesize \ytableausetup{centertableaux, boxsize=1.2em}
	\begin{ytableau}
	\none & b & c \\
	d & e\\
	f
	\end{ytableau}}\ \right) = \sum_{
	\arraycolsep=1.4pt\def\arraystretch{0.8}
	\begin{array}{cccc}
	  & \,\,\,\, m_b & \leq m_c \\
       & \,\,\,\,\,\vsmall \\
	  m_d & \leq m_e\, & \, \\
	  \vsmall & \\
	  m_f & \, & \, 
	\end{array} } \frac{1}{m_b^b \cdot m_c^c \cdot m_d^d \cdot m_e^e \cdot m_f^f}  \,.
\end{equation*}
 First notice that the multiple zeta(-star) values in \eqref{eq:mzv} are given by Schur multiple zeta values of a Young tableau with a single column and row, respectively, i.e.,
\begin{equation}\label{eq:columrow}
\zeta(k_1,\dots,k_r) = \zeta\!\left(\ {\footnotesize \ytableausetup{centertableaux, boxsize=1.2em}
	\begin{ytableau}
	k_1  \\
	\svdots \\
	k_r
	\end{ytableau}}\ \right) \quad \text{ and }\qquad \zeta^\star(k_1,\dots,k_r) =\zeta\!\left(\ {\footnotesize \ytableausetup{centertableaux, boxsize=1.2em}
	\begin{ytableau}
	k_1 & \cdots & k_r
	\end{ytableau}}\ \right) \,.  
\end{equation}
 
In this paper we will give explicit formulas for Schur multiple zeta values of Checkerboard style, by which we mean that the entries of the Young tableaux have two alternating entries. The formula in \eqref{eq:31formula} is an example for such an identity of Checkerboard style Schur multiple zeta values, since it can be written as
 \[ \ytableausetup{mathmode,boxsize=1.2em,aligntableaux=center} 
\zeta
 \left(\ {\footnotesize
 \begin{ytableau}
 1        \\
 *(gray)3 \\
 \svdots    \\
 1        \\
 *(gray)3 
 \end{ytableau}}
 \ \right) = \frac{1}{4^n} \zeta(\{4\}^n) \,.\]
Here and in the following the coloring is just for optical reasons and $n$ denotes the number of blocks $\ytableausetup{mathmode,boxsize=0.9em,aligntableaux=center} {\scriptsize
  \begin{ytableau}
  1        \\
   *(gray)3 
  \end{ytableau}}\,$. As we will see in Theorem \ref{thm:SSstar},
 a similar formula holds for the following skew Young tableaux
 \[ \ytableausetup{mathmode,boxsize=1.2em,aligntableaux=center} 
 \zeta
 \left(\ {\footnotesize
 \begin{ytableau}
 \none    & \none  & 1        \\
 \none    & \adots & *(gray)3 \\
 1        & \adots \\
 *(gray)3 
 \end{ytableau}}
 \ \right) = \frac{1}{4^n} \zeta^\star(\{4\}^n) \,.\]
More surprisingly we will see  in Theorem~\ref{thm:AB13} that, after adding a $1$ on the bottom left or a $3$ on the top right, we can obtain all odd single zeta values by these Schur multiple zeta values:
\begin{equation} \label{eq:oddsums}
\ytableausetup{mathmode,boxsize=1.2em,aligntableaux=center} 
\zeta
\left(\ {\footnotesize
\begin{ytableau}
 \none & \none    & \none  & 1        \\
 \none & \none    & \adots & *(gray)3 \\
 \none & 1        & \adots \\
 1     & *(gray)3  
\end{ytableau}}
\ \right)  = \frac{2}{4^n} \zeta(4n+1) ,\qquad
\zeta\left(
\ {\footnotesize
\begin{ytableau}
 \none  & \none & 1 & *(gray) 3 \\
 \none  & \adots& *(gray) 3 \\
 1 & \adots \\
 *(gray) 3
\end{ytableau}}
\
\right) = \frac{1}{4^{n}} \zeta(4n+3)\,.
\end{equation}
These formulas are valid for all $n\geq 1$ and $n\geq 0$ respectively. In particular this gives new sum representation for odd single zeta values. Using \eqref{eq:oddsums} we will give explicit expressions for various classes of Schur multiple zeta values of Checkerboard style. For example we will show (as a special case of Corollary \ref{cor:31case}) that the value \eqref{eq:3stair} with alternating entries in $1$ and $3$ is given by a Hankel determinant of odd single zeta values:
\begin{align*}
\ytableausetup{mathmode,boxsize=1.2em,aligntableaux=center} 
\zeta\left(\ {\footnotesize
				\begin{ytableau}
				*(gray)3 & 1 & *(gray)3\\
				1 & *(gray)3\\
				*(gray)3
				\end{ytableau}}\ \right) 
					= \frac{1}{4^2}  \begin{vmatrix}
					\zeta(3) & \zeta(7)\\
					\zeta(7) & \zeta(11)
					\end{vmatrix} \,.
\end{align*}
In general we will give formulas for these type of Schur multiple zeta values and their skew type versions. For example the higher cases of above formulas are given by
\begin{align*}
\ytableausetup{mathmode,boxsize=1.2em,aligntableaux=center} 
\zeta\left(\ {\footnotesize 
		\begin{ytableau}
		1 & *(gray)3 & 1 & *(gray)3\\
		*(gray)3 & 1 & *(gray)3\\
		1 & *(gray)3\\
		*(gray)3
		\end{ytableau}}\ \right) 
			&= \frac{1}{4^4} \begin{vmatrix}
			\zeta(7) & \zeta(11)\\
			\zeta(11) & \zeta(15) 
			\end{vmatrix}
		 \,,\quad
		&&\zeta\left(\ {\footnotesize 
		\begin{ytableau}
		*(gray)3 & 1 & *(gray)3 & 1 & *(gray)3\\
		1 & *(gray)3 & 1 & *(gray)3\\
		*(gray)3 & 1 & *(gray)3\\
		1 & *(gray)3\\
		*(gray)3
		\end{ytableau}}\ \right) 
		= \frac{1}{4^6}  \begin{vmatrix}
		\zeta(3) & \zeta(7) & \zeta(11)\\
		\zeta(7) & \zeta(11) & \zeta(15) \\
		\zeta(11) &  \zeta(15) & \zeta(19)
		\end{vmatrix}   \,,
\end{align*}
from which the reader should already be able to guess the general form. Most of our results will be proven for Checkerboard style Schur multiple zeta values with arbitrary entries $a,b$. But because of \eqref{eq:oddsums} the case $(a,b)=(1,3)$ will always give even more explicit formulas. 

\section{Schur multiple zeta values}

\subsection{Notations}
 Throughout the present paper, the set of all positive integers is denoted by $\N = \{1,2,3,\dots\}$.

 A partition of $n\in\N$ is a tuple $\lambda = (\lambda_1,\dots,\lambda_h)$ of positive integers
 $\lambda_1 \geq \dots \geq \lambda_h \geq 1$ with $n = |\lambda|= \lambda_1 + \dots + \lambda_h$.
 For another partition $\mu=(\mu_1,\dots,\mu_r)$ we write $\mu \subset \lambda$ if $r\leq h$ and $\mu_j <\lambda_j$ for $j=1,\dots,r$.  
 For partitions $\lambda,\mu$ with $\mu \subset \lambda$ we identify the pair $\lambda\slash\mu=(\lambda,\mu)$ with its (skew) Young diagram
	 	\[D(\lambda \slash \mu) = \left\{(i,j) \in \Z^2 \mid 1 \leq i \leq h\,, \mu_i < j \leq \lambda_i \right\},\]
 where we set $\mu_j = 0$ for $j>r$.
 In the case where $\mu=\emptyset$ is the empty partition we just write
 $\lambda\slash\mu = \lambda$.
 An entry $(i,j)\in D(\lambda\slash\mu)$ is called a corner of $\lambda \slash \mu$
 if $ (i,j+1) \not\in D(\lambda\slash\mu)$ and $ (i+1,j) \not\in D(\lambda\slash\mu)$.
 We denote the set of all corners of $\lambda \slash \mu$ by $C(\lambda \slash \mu)$.
For example when $\lambda\slash\mu=(5,4,3)\slash (3,1)$ we have $\Cor(\lambda\slash\mu) = \{(1,5),(2,4),(3,3)\}$, which we visualize as $\bullet$ in the corresponding Young diagram:
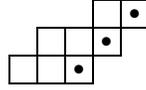
\[{\footnotesize 	\ytableausetup{centertableaux, boxsize=1.2em}\begin{ytableau}
	\none & \none & \none & \, & \bullet \\
	\none & \, & \, & \bullet \\
	\, & \, & \bullet 
	\end{ytableau}}\]

 The conjugate of $\lambda\slash\mu$ is the pair $\lambda^{\prime}\slash\mu^{\prime}$ with  $\lambda^\prime=(\lambda^\prime_1,\dots,\lambda^\prime_{\lambda_1})$ and 
$\mu^\prime=(\mu^\prime_1,\dots,\mu^\prime_{\mu_1})$ whose Young diagram is the transpose of that of $\lambda\slash\mu$.
 For example when $\lambda\slash\mu=(5,4,3)\slash(3,1)$
 we have $\lambda^\prime\slash\mu^{\prime}=(3,3,3,2,1)\slash(2,1,1)$,
 which is visualized by  
	\[\lambda\slash\mu \,\, = \,\, {\footnotesize 
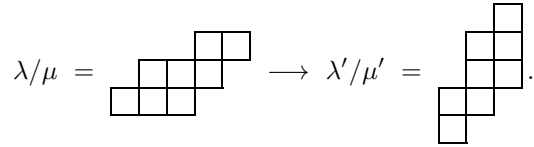
\begin{ytableau}
		\none & \none & \none & \,& \,\\
		\none & \,& \,& \,\\
		\,& \,& 
		\end{ytableau}}
  \,\, \longrightarrow \,\, \lambda^\prime\slash\mu^\prime \,\, = \,\, {\footnotesize \ytableausetup{centertableaux, boxsize=1.2em}
        \begin{ytableau}
		\none & \none & \, \\
		\none & \,& \, \\
     	\none & \, & \, \\
		\,& \,\\
		\,
		\end{ytableau}} \,. \] 

 A (skew) Young tableau $\kk =  (k_{i,j})_{(i,j) \in D(\lambda \slash \mu)}$ of shape $\lambda \slash \mu$ is a filling of $D(\lambda\slash\mu)$ obtained by putting $k_{i,j}\in\N$ into the $(i,j)$-entry of $D(\lambda\slash\mu)$. For shorter notation we will also just write $(k_{i,j})$ in the following if the shape $\lambda \slash \mu$ is clear from the context. 
For example when $\lambda\slash\mu=(5,4,3)\slash (3,1)$ we visualize this Young tableau by
 \[\kk \,\, = \,\, (k_{i,j}) \,\, = \,\, {\footnotesize \ytableausetup{centertableaux, boxsize=1.8em}
\begin{ytableau}
 \none & \none & \none & k_{1,4} & k_{1,5} \\
 \none & k_{2,2} & k_{2,3} & k_{2,4} \\
 k_{3,1} & k_{3,2} & k_{3,3} 
\end{ytableau}}\,. \]
A Young tableau $(m_{i,j})$ is called semi-standard if $m_{i,j}<m_{i+1,j}$ and $m_{i,j}\leq m_{i,j+1}$ for all $i$ and $j$.
The set of all Young tableaux and all semi-standard Young tableaux of shape $\lambda \slash \mu$ are denoted by $T(\lambda \slash \mu)$ and 
$\SSYT(\lambda \slash \mu)$, respectively.

\subsection{Schur multiple zeta values}

We call a Young tableau $\kk = (k_{i,j}) \in T(\lambda \slash \mu)$ admissible if $k_{i,j} \geq 2$ for $(i,j) \in \Cor(\lambda\slash\mu)$.
For an admissible $\kk = (k_{i,j}) \in T(\lambda \slash \mu)$
the Schur multiple zeta value is defined by  
\[ \zeta(\kk) = \sum_{(m_{i,j}) \in \SSYT(\lambda \slash \mu)} \prod_{(i,j) \in D(\lambda \slash \mu)}  \frac{1}{m_{i,j}^{k_{i,j}}} \,. \]
We notice that the condition $k_{i,j} \geq 2$ for the corners ensures the convergence of the above series 
(see \cite[Lemma~2.1]{NakasujiPhuksuwanYamasaki}). By definition it is clear that these numbers generalize multiple zeta and zeta-star values as seen in \eqref{eq:columrow}.
The product of two multiple zeta values can be expressed by using the so called harmonic product formula, which in the lowest depth is given by $\zeta(a) \zeta(b) = \zeta(a,b) + \zeta(b,a) + \zeta(a+b)$. Using the notion of Schur multiple zeta values, this can be expressed even nicer as
\[\zeta\!\left(\ \footnotesize \ytableausetup{centertableaux, boxsize=1.2em}
	\begin{ytableau}
a
	\end{ytableau}\ \right) \zeta\!\left(\ {\footnotesize \ytableausetup{centertableaux, boxsize=1.2em}
		\begin{ytableau}
b
		\end{ytableau}}\ \right)  = \zeta\!\left(\ {\footnotesize \ytableausetup{centertableaux, boxsize=1.2em}
			\begin{ytableau}
		a & b
			\end{ytableau}}\ \right) + \zeta\!\left(\ {\footnotesize \ytableausetup{centertableaux, boxsize=1.2em}
				\begin{ytableau}
			b\\a
				\end{ytableau}}\ \right)  \,.\]
In the following we want to explain this harmonic product for arbitrary Schur multiple zeta values. 
For $\lambda = (\lambda_1,\dots,\lambda_h)$ and $\mu \subset \lambda$ we call $(1,\lambda_1) \in D(\lambda\slash\mu)$ the \emph{top-right entry} and $(h,1) \in D(\lambda\slash\mu)$ the \emph{bottom-left entry} of $\lambda\slash\mu$.
For example when $\lambda\slash\mu=(5,4,3)\slash (3,1)$ 
the top-right entry is $(1,5)$ and the bottom-left entry is $(3,1)$, which are respectively visualized as $\mathsf{tr}$ and $\mathsf{bl}$ in the corresponding Young diagram:\[{\footnotesize 	\ytableausetup{centertableaux, boxsize=1.2em}
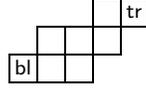
\begin{ytableau}
	\none & \none & \none & \, & \mathsf{tr} \\
	\none & \, & \, & \, \\
	\mathsf{bl} & \, & \, 
	\end{ytableau}}\]
For partitions $\lambda^{1},\lambda^{2},\mu^{1},\mu^{2}$ with $\mu^{1}\subset \lambda^1$ and $\mu^{2}\subset \lambda^2$,
define the horizontal gluing $\lambda^1\slash\mu^1\oplus_{\mathrm{h}} \lambda^2\slash\mu^2$ of $\lambda^1\slash\mu^1$ and $\lambda^2\slash\mu^2$ 
by putting $\lambda^2\slash\mu^2$ next to $\lambda^1\slash\mu^1$ such that the bottom-left entry of $\lambda^2\slash\mu^2$ is on the right of the top-right entry of $\lambda^1\slash\mu^1$.
Moreover, 
define the vertical gluing $\lambda^1\slash\mu^1\oplus_{\mathrm{v}} \lambda^2\slash\mu^2$ of $\lambda^1\slash\mu^1$ and $\lambda^2\slash\mu^2$
by putting $\lambda^2\slash\mu^2$ next to $\lambda^1\slash\mu^1$ such that the bottom-left entry of $\lambda^2\slash\mu^2$ is on top of the top-right entry of $\lambda^1\slash\mu^1$.
Analogously we define for Young diagrams $\kk_1\in T(\lambda^1\slash\mu^1)$ and $\kk_2\in T(\lambda^2\slash\mu^2)$ the Young diagrams $\kk_1\oplus_{\mathrm{h}}\kk_2\in T(\lambda^1\slash\mu^1\oplus_{\mathrm{h}} \lambda^2\slash\mu^2)$
and $\kk_1\oplus_{\mathrm{v}}\kk_2\in T(\lambda^1\slash\mu^1\oplus_{\mathrm{v}} \lambda^2\slash\mu^2)$.

\begin{example} For $\lambda^1\slash\mu^1=(5,4,3)\slash (3,1)$ and $\lambda^2\slash\mu^2=(3,3,3,2,1)\slash (2,1,1)$ the horizontal gluing is given by 
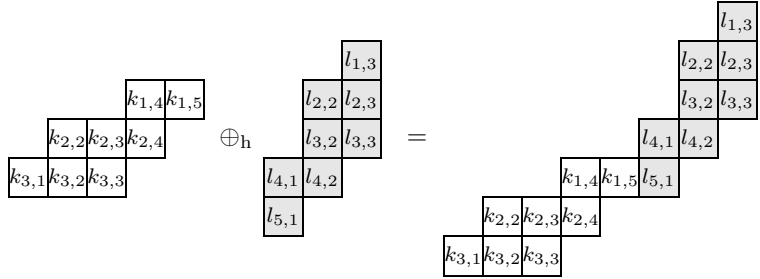
\begin{align*}
{\footnotesize 	\ytableausetup{centertableaux, boxsize=1.7em}\begin{ytableau}
	\none & \none & \none  & k_{1,4} & k_{1,5} \\
	\none & k_{2,2} & k_{2,3} & k_{2,4} \\
	k_{3,1} & k_{3,2} & k_{3,3} 
	\end{ytableau}}
	\ \oplus_{\mathrm{h}} \,%
{\footnotesize 	\ytableausetup{centertableaux, boxsize=1.7em}\begin{ytableau}
	\none & \none & *(gray)l_{1,3} \\
	\none & *(gray)l_{2,2} & *(gray)l_{2,3} \\
	\none & *(gray)l_{3,2} & *(gray)l_{3,3} \\
	*(gray)l_{4,1} & *(gray)l_{4,2} \\
	*(gray)l_{5,1}
	\end{ytableau}}
\ & \,\, = \,\, 
{\footnotesize 	\ytableausetup{centertableaux, boxsize=1.7em}\begin{ytableau}
	\none & \none & \none & \none & \none & \none & \none & *(gray)l_{1,3} \\
	\none & \none & \none & \none & \none & \none & *(gray)l_{2,2} & *(gray)l_{2,3} \\
	\none & \none & \none & \none & \none & \none & *(gray)l_{3,2} & *(gray)l_{3,3} \\
	\none & \none & \none & \none & \none & *(gray)l_{4,1} & *(gray)l_{4,2} \\
	\none & \none & \none  & k_{1,4} & k_{1,5} & *(gray)l_{5,1} \\
	\none & k_{2,2} & k_{2,3} & k_{2,4} \\
	k_{3,1} & k_{3,2} & k_{3,3} 
	\end{ytableau}}
\end{align*}
and the vertical gluing is given by
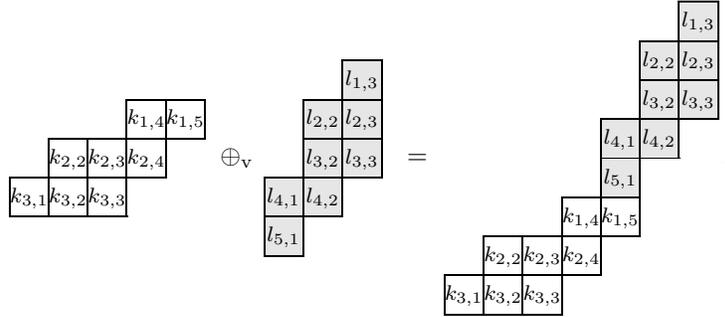
\begin{align*}
{\footnotesize 	\ytableausetup{centertableaux, boxsize=1.7em}\begin{ytableau}
	\none & \none & \none  & k_{1,4} & k_{1,5} \\
	\none & k_{2,2} & k_{2,3} & k_{2,4} \\
	k_{3,1} & k_{3,2} & k_{3,3} 
	\end{ytableau}}
	\ \oplus_{\mathrm{v}} \,%
{\footnotesize 	\ytableausetup{centertableaux, boxsize=1.7em}\begin{ytableau}
	\none & \none & *(gray)l_{1,3} \\
	\none & *(gray)l_{2,2} & *(gray)l_{2,3} \\
	\none & *(gray)l_{3,2} & *(gray)l_{3,3} \\
	*(gray)l_{4,1} & *(gray)l_{4,2} \\
	*(gray)l_{5,1}
	\end{ytableau}}
\ & \,\, = \,\, 
{\footnotesize 	\ytableausetup{centertableaux, boxsize=1.7em}\begin{ytableau}
	\none & \none & \none & \none & \none & \none & *(gray)l_{1,3} \\
	\none & \none & \none & \none & \none & *(gray)l_{2,2} & *(gray)l_{2,3} \\
	\none & \none & \none & \none & \none & *(gray)l_{3,2} & *(gray)l_{3,3} \\
	\none & \none & \none & \none & *(gray)l_{4,1} & *(gray)l_{4,2} \\
	\none & \none & \none & \none & *(gray)l_{5,1} \\
	\none & \none & \none  & k_{1,4} & k_{1,5} \\
	\none & k_{2,2} & k_{2,3} & k_{2,4} \\
	k_{3,1} & k_{3,2} & k_{3,3} 
	\end{ytableau}}\,.	
\end{align*}
\end{example}
 Now the following is immediately from the definition of the Schur multiple zeta values.

\begin{lemma}(Harmonic product)
Let $\kk_1\in T(\lambda^1\slash\mu^1)$ and $\kk_2\in T(\lambda^2\slash\mu^2)$ be admissible indices.
Then both $\kk_1\oplus_{\mathrm{h}}\kk_2$ and $\kk_1\oplus_{\mathrm{v}}\kk_2$ are admissible and   
\begin{equation}
\label{for:harmonicproduct}
\zeta(\kk_1)\zeta(\kk_2)=\zeta(\kk_1\oplus_{\mathrm{h}}\kk_2)+\zeta(\kk_1\oplus_{\mathrm{v}}\kk_2) \,.
\end{equation}
\end{lemma}
Let $T^{\mathrm{diag}}(\lambda/\mu)$ be the subset of $T(\lambda/\mu)$
consisting of all Young tableaux $\kk=(k_{i,j})$ satisfying $k_{i,j} = k_{i^{\prime},j^{\prime}}$ whenever $j-i=j^{\prime}-i^{\prime}$,
which means that $\kk$ has the same entries on each diagonal.
The following result from \cite{NakasujiPhuksuwanYamasaki} states, that the Schur multiple zeta value $\zeta(\kk)$ for $\kk\in T^{\mathrm{diag}}(\lambda/\mu)$ can be written as a determinant of a matrix whose entries are multiple zeta-(star) values.
\begin{thm}[{\cite[Theorem~4.3]{NakasujiPhuksuwanYamasaki}}] \label{thm:JT}
Let $\lambda = (\lambda_1,\dots,\lambda_h)$ and $\mu=(\mu_1,\dots,\mu_r)$ be partitions with $\mu\subset\lambda$ and 
$\kk = (k_{i,j}) \in T^{\mathrm{diag}}(\lambda \slash \mu)$ an admissible Young tableau.
Write $d_m=k_{i,i+m}$ for $m\in \Z$.
\begin{enumerate}[i)]
\item If the last entry of every row in $\kk$ is strictly larger than $1$, then we have
\begin{equation}\label{for:JTzetastar}
\zeta(\kk) = \det\left( \zeta^\star(d_{\mu_j-j+1} , d_{\mu_j-j+2}, \dots, d_{\mu_j-j+(\lambda_i - \mu_j - i +j)}) \right)_{1\leq i, j \leq h} \,, 
\end{equation}
where we set $\zeta^\star(\,\cdots)=1$ if $\lambda_i - \mu_j - i + j =0$ and $\zeta^\star(\,\cdots)=0$ if $\lambda_i - \mu_j - i + j <0$. 
\item If the last entry of every column in $\kk$ is strictly larger than $1$, then we have
\begin{equation}\label{for:JTzeta}
\zeta(\kk) = \det\left( \zeta(d_{-\mu'_j+j-1} , d_{-\mu'_j+j-2}, \dots, d_{-\mu'_j+j-(\lambda'_i - \mu'_j - i +j)}) \right)_{1\leq i, j \leq \lambda_1} \,,
\end{equation}
where we set $\zeta(\,\cdots)=1$ if $\lambda'_i - \mu'_j - i +j =0$ and $\zeta(\,\cdots)=0$ if $\lambda'_i - \mu'_j - i +j <0$. 
\end{enumerate}
\end{thm}
 
\subsection{Checkerboard style Schur multiple zeta values} 
Let $a,b\in\mathbb{N}$ with $b\ge 2$.
In the remaining part of this paper we will be interested in Young diagrams $\lambda/\mu$ satisfying the following properties:
there exists a Young tableau $\kk_{a,b}=(k_{i,j})\in T(\lambda/\mu)$ whose entries are arranged in a "Checkerboard style",
that is, 
\begin{enumerate}[i)]
\item
$k_{i,j} = a$ if $i-j$ is even and $k_{i,j} = b$ otherwise, or, 
$k_{i,j} = a$ if $i-j$ is odd and $k_{i,j} = a$ otherwise.
\item
$k_{i,j}=b$ whenever $(i,j)\in C(\lambda/\mu)$.
\end{enumerate}
Notice that if $\kk_{a,b}$ exists,
 it is in $T^{\mathrm{diag}}(\lambda/\mu)$ by the first condition
and always admissible by the second one. Moreover conditions i) and ii) ensure that $\kk_{a,b}$ is unique for a given Young diagram $\lambda/\mu$.
We call such a Young diagram $\lambda/\mu$ checkerboardable.
The aim of this paper is, for a checkerboardable Young diagram $\lambda/\mu$,
to study the checkerboard style Schur multiple zeta value $\zeta(a,b;\lambda/\mu)=\zeta(\kk_{a,b})$. 

An immediate consequence of  the Jacobi-Trudi formula above is the following.
\begin{lemma}
 Let $\lambda\slash\mu$ be a checkerboardable Young diagram.
\begin{enumerate}[i)]
\item
 If the last entry of every row in $\kk_{a,b}$ is $b$, then we have
\[
 \zeta(a,b;\lambda/\mu)\in\mathbb{Q}[\zeta^{\star}(\{a,b\}^n),\zeta^{\star}(b,\{a,b\}^n)\,|\,n\geq 0] \,.
\]
\item
 If the last entry of every column in $\kk_{a,b}$ is $b$, then we have  
\[
 \zeta(a,b;\lambda/\mu)\in\mathbb{Q}[\zeta(\{a,b\}^n),\zeta(b,\{a,b\}^n)\,|\,n\geq 0] \,.
\]
\end{enumerate} 
\end{lemma}

\begin{example} 
\begin{enumerate}[i)]
\item
 When $\lambda\slash\mu=(5,4,3)\slash(3,1)$ we have from \eqref{for:JTzetastar} 
\begin{align*}
\ytableausetup{mathmode,boxsize=1.2em,aligntableaux=center} 
\zeta(a,b;\lambda\slash\mu)
=\zeta\left(\ {\footnotesize
\begin{ytableau}
\none & \none & \none & a & *(gray)b\\
\none & *(gray)b & a & *(gray)b \\
*(gray)b & a & *(gray)b
\end{ytableau}}
\ \right) &= \left|
\begin{array}{ccc}
\zeta^{\star}(a,b) & \zeta^{\star}(b,\{a,b\}^2) & \zeta^{\star}(b,\{a,b\}^3) \\
1 & \zeta^{\star}(b,a,b) & \zeta^{\star}(b,\{a,b\}^2) \\
0 & \zeta^{\star}(b) & \zeta^{\star}(b,a,b)
\end{array}
\right|\,\\
&=\zeta^{\star}(a,b)\zeta^{\star}(b,a,b)^2+\zeta^{\star}(b)\zeta^{\star}(b,\{a,b\}^3)\\
&\ \ \ -\zeta^{\star}(b,a,b)\zeta^{\star}(b,\{a,b\}^2)-\zeta^{\star}(b)\zeta^{\star}(a,b)\zeta^{\star}(b,\{a,b\}^2)\,.
\end{align*}
\item
 When $\lambda\slash\mu=(3,3,3,2,1)\slash(2,1,1)$ we have from \eqref{for:JTzeta} 
\begin{align*}
\ytableausetup{mathmode,boxsize=1.2em,aligntableaux=center} 
\zeta(a,b;\lambda\slash\mu)
=\zeta\left(\ {\footnotesize 
\begin{ytableau}
 \none & \none & *(gray)b \\
 \none & *(gray)b & a \\
 \none & a & *(gray)b \\
 a & *(gray)b \\
 *(gray)b 
\end{ytableau}}
\ \right) &= \left|
\begin{array}{ccc}
\zeta(a,b) & \zeta(b,\{a,b\}^2) & \zeta(b,\{a,b\}^3) \\
1 & \zeta(b,a,b) & \zeta(b,\{a,b\}^2) \\
0 & \zeta(b) & \zeta(b,a,b)
\end{array}
\right| \,\\
&=\zeta(a,b)\zeta(b,a,b)^2+\zeta(b)\zeta(b,\{a,b\}^3)\\
&\ \ \ -\zeta(b,a,b)\zeta(b,\{a,b\}^2)-\zeta(b)\zeta(a,b)\zeta(b,\{a,b\}^2)\,.
\end{align*}
\end{enumerate}
\end{example}

\section{Ribbons}

 A Young diagram $\lambda\slash\mu$ is called a ribbon if it is connected and does not contain any $2\times 2$ blocks. For example the following Young tableau is an example of a Checkerboard style ribbon.
 \begin{align*}
{\footnotesize
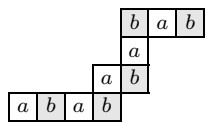
\begin{ytableau}
\none&\none & \none & \none & *(gray)b & a&*(gray)b\\
\none&\none & \none & \none & a &\none\\
\none&\none & \none & a & *(gray)b \\
a&*(gray)b & a & *(gray)b
\end{ytableau}} 
\end{align*}
In this section we will show that all Schur multiple zeta values of Checkerboard style ribbons can be reduced to Schur multiple zeta values of certain types of ribbons, which we call primitive ribbons. First we will define these primitive ribbons and prove basic properties of them. After this we will give explicit evaluating in terms of multiple zeta values in the case $(a,b)=(1,3)$. At the end of this section we consider hooks and anti-hooks. One of the results of this section will be the following.

\begin{thm}\label{thm:13rib}
For $(a,b)=(1,3)$ the space spanned by all Schur multiple zeta values of Checkerboard style ribbons is given by $\Q[\pi^4,\zeta(3),\zeta(5),\zeta(7),\dots]$.
\end{thm}
\subsection{Primitive ribbons} \label{sec:primrib}
 In this section we will 
 let $\delta_n=(n,n-1,\ldots,2,1)$
and set $\delta_0=\delta_{-1}=\emptyset$. 
We now focus on the following checkerboard style Schur multiple zeta values 
 $A(n)=A_{a,b}(n)$ for $n\ge 1$ and $B(n)=B_{a,b}(n)$ for $n\ge 0$.
\begin{align} 
\ytableausetup{mathmode,boxsize=1.1em,aligntableaux=center} 
\label{def:A}
 A(n)
&=A_{a,b}(n)
=
\zeta
\left(
\ {\footnotesize
\begin{ytableau}
 \none & \none    & \none  & a        \\
 \none & \none    & \adots & *(gray)b \\
 \none & a        & \adots \\
 a     & *(gray)b  
\end{ytableau}}
\ \right)
=\zeta(a,b;(n+1,n+1,n,n-1,\ldots,3,2)/\delta_{n})
\,,\\
\label{def:B}
 B(n)
&=B_{a,b}(n)
=
\zeta\left(
\ {\footnotesize
\begin{ytableau}
 \none  & \none & a & *(gray)b \\
 \none  & \adots& *(gray)b \\
 a & \adots \\
 *(gray)b
\end{ytableau}}
\
\right)
=\zeta(a,b;\delta_{n+1}/\delta_{n-1}) \,.
\end{align}
 For convenience, we set $A(n)=0$ if $n\leq 0$ and $B(n)=0$ if $n<0$. 
Further we define for  $n\geq 1$   
\begin{align*}
 L(n)&=L_{a,b}(n)=\zeta(a,b;(2^{2n})/(1^{2n-1}))\,, & 
 L^{\star}(n)&=L^{\star}_{a,b}(n)=\zeta(a,b;((2n)^2)/(2n-1))\,, 
\end{align*}
and for $n \geq 0$
\begin{align*}
 S(n)&=S_{a,b}(n)=\zeta(a,b;(n,n,n-1,\ldots,2,1)/\delta_{n-1})\,, & 
 S^{\star}(n)&=S^{\star}_{a,b}(n)=\zeta(a,b;(n+1,n,\ldots,3,2)/\delta_{n-1})\,, 
\end{align*}
where we set $S(0) = S^\star(0)=1$.
In terms of Young tableau these are given by 
\begin{align*}
\ytableausetup{mathmode,boxsize=1.1em,aligntableaux=center} 
 L(n)=
\zeta
\left(
\ {\footnotesize
\begin{ytableau}
 \none & a        \\
 \none & *(gray)b \\
 \none & \svdots  \\
 \none & a        \\
 a     & *(gray)b 
\end{ytableau}}
 \ \right) \,,
 & &
 L^{\star}(n)&=
\zeta
\left(
\ {\footnotesize
\begin{ytableau}
 \none & \none    & \none  & \none & a       \\
 a     & *(gray)b & \cdots & a     & *(gray)b
 \end{ytableau}}
 \ \right) \,,
\\[5pt]
 S(n)=
\zeta
\left(
\ {\footnotesize
\begin{ytableau}
 \none    & \none  & a        \\
 \none    & \adots & *(gray)b \\
 a        & \adots \\
 *(gray)b 
\end{ytableau}}
\ \right) \,,
 & &
 S^{\star}(n)&=
\zeta
\left(
\ {\footnotesize
\begin{ytableau}
 \none & \none  & a      & *(gray)b \\
 \none & \adots & \adots \\
 a     & *(gray)b 
\end{ytableau}}
\ \right)\,.
\end{align*}
 Notice that $A(n)$ and $B(n)$ are obtained
 by adding an $a$ on the bottom left of $S(n)$ and a $b$ on the top right, respectively.
 Using the harmonic product \eqref{for:harmonicproduct}, one easily show by induction on $n$ that  
 both $A(n)$ and $B(n)$ satisfy the following recursion formulas.

\begin{lemma}\label{lem:ABrec}
\begin{enumerate}[i)]
\item
 For $n\ge 2$, we have  
\begin{align}
\label{for:Arec}
 A(n)&=(-1)^{n-1}L(n)-\sum^{n-1}_{k=1}(-1)^{n-k}A(k)\,\zeta(\{a,b\}^{n-k})\,,\\
\label{for:Arecstar}
 A(n)&=(-1)^{n-1}L^{\star}(n)-\sum^{n-1}_{k=1}(-1)^{n-k}A(k)\,\zeta^\star(\{a,b\}^{n-k})\,.
\end{align} 
\item
 For $n\ge 1$, we have   
\begin{align}
\label{for:Brec}
 B(n)&=(-1)^{n} \zeta(b,\{a,b\}^{n})-\sum^{n-1}_{k=0}(-1)^{n-k}B(k)\,\zeta(\{a,b\}^{n-k})\,,\\
\label{for:Brecstar}
 B(n)&=(-1)^{n} \zeta^\star(b,\{a,b\}^{n})-\sum^{n-1}_{k=0}(-1)^{n-k}B(k)\,\zeta^\star(\{a,b\}^{n-k})\,.
\end{align} 
\end{enumerate} 
\end{lemma}
 
%
 Similarly there is a connection between the Schur multiple zeta values $S(n)$ and $\zeta(\{a,b\}^{n})$.

\begin{lemma}\label{lem:SSstargenseries}
\begin{enumerate}[i)]
\item The generating series of $S(n)$ and $S^{\star}(n)$ are given by
\begin{equation}
\label{for:gf_of_SSstar}
\sum_{n\geq 0} S(n) x^n = \left( \sum_{n\geq 0} (-1)^n \zeta(\{a,b\}^{n}) x^n\right)^{-1} \,,\quad
\sum_{n\geq 0} S^{\star}(n) x^n = \left( \sum_{n\geq 0} (-1)^n \zeta^\star(\{a,b\}^{n}) x^n\right)^{-1} \,,
\end{equation}
\item For all $n\geq 0$ we have 
\begin{equation}
\label{for:SstarS}
 S^{\star}(n) = \sum_{k=0}^n S(k)\zeta(\{a+b\}^{n-k}) \,.
\end{equation}
\end{enumerate}
\end{lemma}
\begin{proof}
 To prove \eqref{for:gf_of_SSstar}, it is sufficient to show
\[
 \sum^{n}_{k=0}(-1)^{n-k}S(k) \zeta(\{a,b\}^{n-k})
=
\begin{cases}
1 & n=0, \\
0 & \text{otherwise},
\end{cases}
\quad 
 \sum^{n}_{k=0}(-1)^{n-k}S^{\star}(k) \zeta^\star(\{a,b\}^{n-k})
=
\begin{cases}
1 & n=0, \\
0 & \text{otherwise}.
\end{cases}
\] 
 This is proven by inductively on $n$ with the help of \eqref{for:harmonicproduct}.
 To prove \eqref{for:SstarS} one uses the formula
\[\zeta^\star(\{a,b\}^n) = \sum_{k=0}^n \zeta(\{a,b\}^k) \zeta^{\star}(\{a+b\}^{n-k}) \,,\]
 which is obtained in \cite[Theorem 6]{Muneta2008} together with \eqref{for:gf_of_SSstar} and 
 the well-known formula
\[
 \sum_{n\geq 0}(-1)^n\zeta^{\star}(\{a+b\}^n)x^n
=\left(\sum_{n\geq 0}\zeta(\{a+b\}^n)x^n\right)^{-1} \,.
\]
\end{proof} 

\subsection{Primitive ribbons for $(a,b)=(1,3)$}

In the case $(a,b) = (1,3)$ the Schur multiple zeta values $A,B,L$ and $S$ can be evaluated explicitly in terms of multiple zeta values. 
\begin{thm}
\label{thm:SSstar}
\begin{enumerate}[i)]
\item
 For $n\ge 1$, we have  
\begin{align}
\label{for:S13}
\ytableausetup{mathmode,boxsize=1.2em,aligntableaux=center} 
 S_{1,3}(n)
&=\zeta
\left({\footnotesize
\ 
\begin{ytableau}
 \none    & \none     & 1        \\
 \none    & \adots    & *(gray)3 \\
 1        & \adots \\
 *(gray)3 
\end{ytableau}}
\ \right) 
= \frac{1}{4^n} \zeta^\star(\{4\}^n) \,.
\end{align}
\item
 For $n\ge 1$, we have  
\begin{align}
\label{for:Sstar13}
\ytableausetup{mathmode,boxsize=1.2em,aligntableaux=center} 
 S^{\star}_{1,3}(n)
&=\zeta
\left({\footnotesize
\ 
\begin{ytableau}
 \none        & \none  & 1      & *(gray)3 \\
 \none        & \adots & \adots \\
 1 & *(gray)3 
\end{ytableau}}
\ \right)
 = \sum_{k=0}^n \frac{1}{4^k} \zeta^\star(\{4\}^k) \zeta(\{4\}^{n-k})
 \,.
\end{align}
\end{enumerate} 
\end{thm}
\begin{proof}
 It is easy to see that \eqref{for:Sstar13} is derived from \eqref{for:S13} together with \eqref{for:SstarS}.
 Hence we just need to prove \eqref{for:S13}.  For $k_1,\dots,k_r \geq 1$ define a multiple polylogarithm by 
\[ L(k_1,\dots,k_r;x) = \sum_{0 < m_1 < \dots < m_r} \frac{x^{m_r}}{m_1^{k_1} \dots m_r^{k_r} }\,.\] 
From \cite[Theorem~11.1]{BBBL}, where a different order of summation is used, it is known that 
\begin{align}\label{eq:geneL13} \sum_{n\ge 0} L(\{1,3\}^n;x) t^{4n} = \hg{\frac{t}{2}(1+i)}{-\frac{t}{2}(1+i)}{1}{x}  \hg{\frac{t}{2}(1-i)}{-\frac{t}{2}(1-i)}{1}{x} \,,
\end{align}
where $\hg{a}{b}{c}{x} = \sum_{n\geq 0} \frac{(a)_n (b)_n}{(c)_n} \frac{x^n}{n!} $ is the usual Gauss hypergeometric function. In particular setting $x=1$, $t=(1+i)y$ and using the well-known formula
\begin{equation*}
\hg{a}{b}{c}{1} = \frac{\Gamma(c) \Gamma(c-a-b)}{\Gamma(c-a) \Gamma(c-b)} \,
\end{equation*}
with $\Gamma(x)$ being the gamma function,
we obtain  
\begin{equation}\label{eq:gen4}
 \sum_{n\ge 0} (-4)^n \zeta(\{1,3\}^n) y^{4n}  =  \frac{1}{\Gamma(1-y) \Gamma(1+y) \Gamma(1-iy) \Gamma(1+iy) } \,. 
\end{equation} 
 Notice that $\zeta(\{1,3\}^n)= L(\{1,3\}^n;1)$.
 Therefore, letting $(a,b)=(1,3)$ and replacing $x$ by $4x^4$ in the first equation of \eqref{for:gf_of_SSstar}
 we see that by \eqref{eq:gen4}
\begin{align}
\label{for:Sgene}
\begin{split}
 \sum_{n\geq 0} 4^n S_{1,3}(n) x^{4n}
 & = \left( \sum_{n\geq 0} (-4)^n \zeta(\{1,3\}^n) x^{4n}\right)^{-1}  = \Gamma(1-x) \Gamma(1+x) \Gamma(1-ix) \Gamma(1+ix)\\
 & = \frac{\pi x}{\sin(\pi x)}\frac{\pi x}{\sinh(\pi x)} 
  = \prod_{m>0} \frac{1}{1 - \frac{x^4}{m^4}}  =   \sum_{n\geq 0}  \zeta^\star(\{4\}^n) x^{4n} \,.
  \end{split}
\end{align}
\end{proof}

\begin{thm}
\label{thm:AB13}
\begin{enumerate}[i)]
\item
 For $n\ge 1$, we have  
\begin{align}
\label{for:A13}
\ytableausetup{mathmode,boxsize=1.2em,aligntableaux=center} 
 A_{1,3}(n)
=\zeta
\left(
\ {\footnotesize
\begin{ytableau}
 \none & \none    & \none  & 1        \\
 \none & \none    & \adots & *(gray)3 \\
 \none & 1        & \adots \\
 1     & *(gray)3  
\end{ytableau}}
\ \right)
=\frac{2}{4^n} \zeta(4n+1)\,.
\end{align} 
\item
 For $n\ge 0$, we have   
\begin{align}
\label{for:B13}
\ytableausetup{mathmode,boxsize=1.2em,aligntableaux=center} 
 B_{1,3}(n)
=\zeta
\left(
\ {\footnotesize
\begin{ytableau}
 \none   & \none  & 1        & *(gray)3 \\
 \none   & \adots & *(gray)3 \\
 1       & \adots \\
 *(gray)3
\end{ytableau}}
\ \right) 
=\frac{1}{4^{n}} \zeta(4n+3)\,.
\end{align} 
\end{enumerate} 
\end{thm}

\begin{proof}
 We first show \eqref{for:B13}.
 Since 
\begin{align*}
\zeta(3,\{1,3\}^n)
=\frac{1}{4^n}\sum^{n}_{k=0}(-1)^k\zeta(4k+3)\zeta(\{4\}^{n-k})
=\sum^{n}_{k=0}\left(-\frac{1}{4}\right)^k\zeta(4k+3)\zeta(\{1,3\}^{n-k}) \,,
\end{align*}
 where the first equality is obtained in \cite[Theorem~1]{BowmanBradley2003}
 and the second one by the well known identity $\zeta(\{4\}^n)=4^n\zeta(\{1,3\}^n)$, 
 we have from \eqref{for:Brec}  
\begin{align*}
 B_{1,3}(n)
&=(-1)^{n} \zeta(3,\{1,3\}^n)-\sum^{n-1}_{k=0}(-1)^{n-k}B_{1,3}(k) \zeta(\{1,3\}^{n-k})\\
&=\frac{1}{4^n}\zeta(4n+3)-\sum^{n-1}_{k=0}(-1)^{n-k}\left\{B_{1,3}(k)-\frac{1}{4^k}\zeta(4k+3)\right\} \zeta(\{1,3\}^{n-k}) \,.
\end{align*} 
 This shows the desired equation by induction on $n$.
 
 We next show \eqref{for:A13}.
 For $n\geq 0$ and $k \geq 0$, define the truncated version of $A(n)$ by 
\[A(n;k) =  \sum_* \frac{1}{p_0 q_1^3 p_1 \dots q_n^3 p_n} \,,\]
where the sum runs over all $0 < p_0 < k$ and $0 < p_{j-1} \leq q_j > p_j$ for $j=1,\dots,n$.
Define the ($4^n$-weighted) generating function of $A(n;k)$ by
\begin{align*}
G(x,y) 
= \sum_{\substack{n\geq 0\\ k\geq 1}} 4^{n} A(n;k) \, x^k y^{4n} \,.
\end{align*}
Since we are only interested in $A(n;k)$ for $n\ge 1$,
we further put  
\begin{align*}
 \tilde{G}(x,y)=G(x,y)-G(x,0) 
= \sum_{\substack{n\geq 1\\ k\geq 1}} 4^{n} A(n;k) \, x^k y^{4n} \,.
\end{align*}
Then, since $\lim_{k\to\infty}A(n;k)=A_{1,3}(n)$ for $n\geq 1$, we have   
\begin{equation}
\label{for:limittildeG}
 \lim_{x\to 1}(x-1)\tilde{G}(x,y)=\sum_{n\ge 1}4^nA_{1,3}(n)y^{4n} \,.
\end{equation}
In the following we will calculate the limit in the left-hand side of \eqref{for:limittildeG}.

For $k\geq 1$ and $k_1,\dots,k_r \geq 1$,
define a truncated version of the multiple zeta value $\zeta(k_1,\ldots,k_r)$ by
\[
 \zeta_k(k_1,\ldots,k_r)=\sum_{0<m_1<\cdots<m_r<k} \frac{1}{m_1^{k_1}\cdots m_r^{k_r}} \,.
\]
With this it is easy to check from \eqref{for:harmonicproduct} again that
\begin{align*}
A(n;k) = \sum_{j=0}^n S_{1,3}(n)  \cdot (-1)^{n-j} \zeta_k(\{1,3\}^{n-j},1) \,.
\end{align*}
Hence, together with \eqref{for:Sgene}, we obtain
\begin{align}
 G(x,y) 
&= \sum_{n\geq 0} 4^{n} S_{1,3}(n) y^{4n} \cdot  \sum_{\substack{n\geq 0\\ k\geq 1}}  (-4)^{n} \zeta_k(\{1,3\}^{n},1) \,x^k y^{4n} \nonumber \\
\label{eq:GH}
&= \Gamma(1-y) \Gamma(1+y) \Gamma(1-iy) \Gamma(1+iy)  \cdot K(x,y)\,,
\end{align}
where 
\[
 K(x,y)=\sum_{\substack{n\geq 0\\ k\geq 1}}  (-4)^{n} \zeta_k(\{1,3\}^{n},1) \,x^k y^{4n} \,.
\]
We observe that
\[  \theta_x^3  L(\{1,3\}^n;x)  = \sum_{k \geq 1} \zeta_k(\{1,3\}^{n-1},1)\, x^k\,,  \]
where $\theta_x = x \frac{d}{dx}$.
This, by the virtue of \eqref{eq:geneL13}, shows 
\begin{align*}
 K(x,y)
& = -\frac{1}{4y^4} \theta_x^3 \,\left(  \hg{y}{-y}{1}{x} \hg{iy}{-iy}{1}{x} \right)\\
& = \frac{1}{2y} \frac{x}{1-x}  \Big( \hg{1-y}{y}{1}{x} \hg{-iy}{iy}{1}{x} \\
& \qquad - i \hg{1-iy}{iy}{1}{x} \hg{-y}{y}{1}{x} +(i-1) \hg{-y}{y}{1}{x} \hg{-iy}{iy}{1}{x}  \Big) \,\\
&=\frac{1}{2} \frac{x^2}{1-x} \Big( \hg{1-y}{1+y}{2}{x} \hg{-iy}{iy}{1}{x} +\hg{1-iy}{1+iy}{2}{x} \hg{-y}{y}{1}{x}  \Big) \,.
\end{align*}
 Notice that the second equality follows from a direct calculation and the last one from the identity 
\[\hg{1-y}{y}{1}{x}-\hg{-y}{y}{1}{x} = xy \hg{1-y}{1+y}{2}{x} \,.\] 
Now, from \eqref{eq:GH} with the above expression of $K(x,y)$, we obtain   
\[ G(x,y) = \frac{x^2}{2(1-x)} \left( F_1(x,y) F_0(x,iy) + F_1(x,iy) F_0(x,y) \right) \,, \]
 where $F_j(x,y) = \Gamma(1-y)\Gamma(1+y) \hg{j-y}{j+y}{j+1}{x}$ for $j=0,1$. 
Therefore, because $G(x,0) = -\frac{x}{(1-x)} \log(1-x)$, we have 
\begin{equation}
\label{for:tildeG}
 \tilde{G}(x,y)= \frac{x}{2(1-x)} \Bigl( xF_1(x,y) F_0(x,iy) + \log(1-x) + xF_1(x,iy) F_0(x,y) +  \log(1-x)  \Bigr) \,.
\end{equation}
With this expression, we can calculate the desired limit.
To do that, let us recall the expansion due to Ramanujan (see for example \cite[Corollary 20]{E});
\begin{equation}\label{eq:asymptF}
\frac{\Gamma(a)\Gamma(b)}{\Gamma(a+b)} \hg{a}{b}{a+b}{x} = -\log(1-x) - 2\gamma -\psi(a) - \psi(b) + O((1-x)\log(1-x))
\end{equation} 
as $x \rightarrow 1$, where
\[\psi(1+y) = \frac{\Gamma'(1+y)}{\Gamma(1+y)} =  -\gamma + \sum_{n\ge 2}(-1)^n \zeta(n) y^{n-1}\]
denotes the digamma function with $\gamma$ being the Euler-Mascheroni constant.
Setting $a=1-y$ and $b=1+y$ in \eqref{eq:asymptF} we obtain
\begin{align*}
F_1(x,y)&=\frac{\Gamma(1-y)\Gamma(1+y)}{\Gamma(2)} \hg{1-y}{1+y}{2}{x} \\
&= -\log(1-x) - H(y) + O\left((1-x)\log(1-x)\right) 
\end{align*}
as $x \rightarrow 1$, where 
\begin{align*}
H(y) &=2 \gamma+ \psi(1-y) + \psi(1+y)  = \sum_{n\ge 2} \left( (-1)^n  - 1 \right) \zeta(n) y^{n-1} = -2 \sum_{n\geq 1} \zeta(2n+1) y^{2n} \,.
\end{align*}
 Therefore, from the expression \eqref{for:tildeG},
 noticing that $F_0(1,y) = 1$, we have
\begin{align*}
(1-x)\tilde{G}(x,y) = -\frac{1}{2}\left( H(y) + H(iy) \right) + O\left((1-x)\log(1-x)\right) 
\end{align*}
as $x \rightarrow 1$, which implies that 
\[
 \lim_{x\to 1}(x-1)\tilde{G}(x,y)=2\sum_{n\ge 1} \zeta(4n+1) y^{4n}
\] 
 because $H(y)+H(iy) = -4 \sum_{n\ge 1} \zeta(4n+1) y^{4n}$.
 This completes the proof.
\end{proof}
 
\begin{cor}
\label{cor:XXsYYs}
\begin{enumerate}[i)]
\item
 For $n\ge 1$, we have  
\begin{align}
\label{for:X13}
\ytableausetup{mathmode,boxsize=1.2em,aligntableaux=center} 
 L_{1,3}(n)
&=\zeta
\left(
\ {\footnotesize
\begin{ytableau}
 \none & 1        \\
 \none & *(gray)3 \\
 \none & \svdots  \\
 \none & 1        \\
 1     & *(gray)3
\end{ytableau}}
\ \right)
= -2 \sum_{k=1}^{n} \left(-\frac{1}{4}\right)^k \zeta(4k+1) \zeta(\{1,3\}^{n-k}) \,,\\
\label{for:Xstar13}
 L^{\star}_{1,3}(n)
&=\zeta
\left(
\ {\footnotesize
\begin{ytableau}
 \none & \none    & \none  & \none & 1 \\
 1     & *(gray)3 & \cdots & 1     & *(gray)3
\end{ytableau}}
\ \right)
= -2 \sum_{k=1}^{n} \left(-\frac{1}{4}\right)^k \zeta(4k+1) \zeta^{\star}(\{1,3\}^{n-k}) \,.
\end{align} 
\item
 For $n\ge 1$, we have  
\begin{align}
\label{for:Y13}
\ytableausetup{mathmode,boxsize=1.2em,aligntableaux=center} 
\zeta(3,\{1,3\}^n)
&=\zeta
\left(
\ {\footnotesize
\begin{ytableau}
 *(gray)3 \\
 1        \\
 *(gray)3 \\
 \svdots  \\
 1        \\
 *(gray)3
\end{ytableau}}
\ \right)
= \sum_{k=0}^{n} \left(-\frac{1}{4}\right)^k \zeta(4k+3) \zeta(\{1,3\}^{n-k}) \,,\\
\label{for:Ystar13}
\zeta^\star(3,\{1,3\}^n)
&=\zeta
\left(
\ {\footnotesize
\begin{ytableau}
 *(gray)3 & 1 & *(gray)3 & \cdots & 1 & *(gray)3
\end{ytableau}}
\ \right)
= \sum_{k=0}^{n} \left(-\frac{1}{4}\right)^k \zeta(4k+3) \zeta^{\star}(\{1,3\}^{n-k}) \,.
\end{align} 
\end{enumerate} 
\end{cor}
\begin{proof}
 The equations \eqref{for:X13}, \eqref{for:Xstar13} and \eqref{for:Ystar13}
 follow from Theorem \ref{thm:AB13} together with \eqref{for:Arec}, \eqref{for:Arecstar} and \eqref{for:Brecstar} in Lemma \ref{lem:ABrec}, respectively.
 Equation \eqref{for:Y13} is obtained in \cite[Theorem~1]{BowmanBradley2003}.
\end{proof}

\begin{cor} For all $n\geq 1$ we have the following identity between multiple zeta values
\begin{align} \label{eq:relofmzv}
\sum_{j=0}^{n-1} \left( \zeta(\{1,3\}^j,1,2,2,\{1,3\}^{n-j-1}) + 3  \zeta(\{1,3\}^j,1,\{1,3\}^{n-j})\right) &= \frac{2}{4^n} \sum_{j=0}^{n-1} \zeta(\{4\}^j,5,\{4\}^{n-j-1}) \,.
\end{align}
\end{cor}
\begin{proof} Using $\zeta(\{1,3\}^n) = \frac{1}{4^n} \zeta(\{4\}^n)$ and the usual harmonic product of multiple zeta values it is easy to see that the right-hand side of \eqref{eq:relofmzv} is exactly the right-hand side of \eqref{for:X13}. The left-hand side of \eqref{for:X13} can be written in terms of multiple zeta values by using the iterated integral expression in \cite[(6.3)]{NakasujiPhuksuwanYamasaki} or \cite{KanekoYamamoto} (see also Remark \ref{rem:integral} below). Writing these in terms of multiple zeta values gives the left-hand side of \eqref{eq:relofmzv}. 
\end{proof}

\subsection{Hook types}

 We first consider the case where $\lambda$ is a hook $(m,1^n)$ with $m,n$ being non negative integers with $m\ge 1$.
 Notice that $(m,1^{n})$ with $n\geq 1$ is checkerboardable if and only if $m\not\equiv n$ $\pmod{2}$ and therefore there are just two different types of these checkerboardable hooks: $(\text{even},  1^\text{odd})$ and $(\text{odd}, 1^\text{even})$. For example, 
\begin{equation}
\label{eq:A12B12}
\ytableausetup{mathmode,boxsize=1.1em,aligntableaux=center} 
 \zeta(a,b;(4,1^3))
=
\zeta
\left(
\ 
\begin{ytableau}
 a & *(gray)b & a & *(gray)b \\
 *(gray)b \\
 a \\
 *(gray)b 
\end{ytableau}
\ \right)
, \quad 
 \zeta(a,b;(3,1^4))
=
\zeta
\left(
\ 
\begin{ytableau}
 *(gray) b & a & *(gray) b \\
 a \\
 *(gray)b \\
 a \\
 *(gray)b 
\end{ytableau}
\ \right)
\,.
\end{equation}

Recall that by Lemma \ref{lem:ABrec} we have
\begin{align}\label{eq:zbabzsbab}\begin{split}
\zeta^\star(b,\{a,b\}^n) &= \sum_{k=0}^n (-1)^{k} B_{a,b}(k) \zeta^\star(\{a,b\}^{n-k}) \,. \\
\zeta(b,\{a,b\}^n) &= \sum_{k=0}^n (-1)^{k} B_{a,b}(k) \zeta(\{a,b\}^{n-k})\,, \\
\end{split}
\end{align}

The following Proposition states an analogue result of \eqref{eq:zbabzsbab} for hooks, which can be seen as an interpolation between $\zeta^\star(b,\{a,b\}^n)$ and $\zeta(b,\{a,b\}^n)$. 

\begin{prop} \label{prop:hook}
\begin{enumerate}[i)]
\item
 For $p\ge 0$ and $q\ge 1$ we have 
\begin{equation*}
\zeta(a,b;(2p+2,1^{2q-1})  = - \sum_{\substack{0 \leq k_1 \leq p \\[2pt]1 \leq k_2 \leq q}} (-1)^{k_1+k_2} B_{a,b}(k_1+k_2) \zeta^\star(\{a,b\}^{p-k_1}) \zeta(\{a,b\}^{q-k_2}).
\end{equation*}
\item
 For $p,q\ge 0$ we have 
\begin{equation*}
\zeta(a,b;(2p+1,1^{2q}))= \sum_{\substack{0 \leq k_1 \leq p \\[2pt] 0 \leq k_2 \leq q}} (-1)^{k_1+k_2} B_{a,b}(k_1+k_2) \zeta^\star(\{a,b\}^{p-k_1}) \zeta(\{a,b\}^{q-k_2}).
\end{equation*}
\end{enumerate}
\end{prop}
\begin{proof}
 We only show the first assertion, since the proof for the second one is similar.
By using the harmonic product \eqref{for:harmonicproduct}, one obtains the following recursion formula
\[\zeta(a,b; (2p+2,1^{2q-1} )) - \zeta(a,b; (2p, 1^{2q+1})) = \zeta^\star(\{a,b\}^q) \zeta(b,\{a,b\}^p) - \zeta(\{a,b\}^p) \zeta^\star(b,\{a,b\}^q)\,. \]
From this we deduce
 \[\zeta(a,b;(2p+2,1^{2q-1}))
=\sum^{p}_{j=0}\zeta\left(\{a,b\}^{p+q-j}\right)\zeta^{\star}\left(b,\{a,b\}^{j}\right)
-\sum^{p+q}_{j=p}\zeta^{\star}\left(\{a,b\}^{p+q-j}\right)\zeta\left(b,\{a,b\}^{j}\right) =: A_1 - A_2 \,.\]
Using the first equation in $\eqref{eq:zbabzsbab}$ for $n=j$ and making the variable change $k_1 = p+q-j$, $k_2 = n-k$ we obtain for $A_1$
\[A_1 = \sum^{p+q}_{k_2=q}\sum^{p+q-k_2}_{k_1=0} (-1)^{p+q-k_1-k_2} B_{a,b}(p+q-k_1-k_2) \zeta^\star(\{a,b\}^{k_1}) \zeta(\{a,b\}^{k_2}) \,.\]
For $A_2$ we obtain with the first equation in $\eqref{eq:zbabzsbab}$ for $n=p+j$ by setting $k_2 = q-1-j$ and $k_1 = p+j-k$
\[A_2 = \left(\sum^{q-1}_{k_2=0}\sum^{p}_{k_1=0}+\sum^{p+q}_{k_2=q}\sum^{p+q-k_2}_{k_1=0}\right)
 (-1)^{p+q-k_1-k_2} B_{a,b}(p+q-k_1-k_2) \zeta^\star(\{a,b\}^{k_1}) \zeta(\{a,b\}^{k_2}) \,.\]
And therefore $A_1-A_2$ gives the desired identity for $\zeta(a,b;(2p+2,1^{2q-1}))$ after changing $k_1$ to $p-k_1$ and $k_2$ to $q-k_2$.
\end{proof}
In the special case $(a,b)=(1,3)$ we can use Theorem \ref{thm:AB13} ii), i.e. $B_{1,3}(n)
=\frac{1}{4^n}\zeta(4n+3)$, to get
\begin{align*}
\zeta(1,3;(2p+2,1^{2q-1}) )  &= \sum_{\substack{0 \leq k_1 \leq p\\[2pt] 1 \leq k_2 \leq q}} \left(-\frac{1}{4}\right)^{k_1+k_2} \zeta(4(k_1+k_2)+3) \zeta^\star(\{1,3\}^{p-k_1}) \zeta(\{1,3\}^{q-k_2}),\\
\zeta(1,3;(2p+1,1^{2q})) &= - \sum_{\substack{0 \leq k_1 \leq p\\[2pt] 0 \leq k_2 \leq q}} \left(-\frac{1}{4}\right)^{k_1+k_2} \zeta(4(k_1+k_2)+3) \zeta^\star(\{1,3\}^{p-k_1}) \zeta(\{1,3\}^{q-k_2}).
\end{align*}

\begin{example}
By using the explicit evaluating of $\zeta^{\star}(\{1,3\}^n)$  in \cite[Theorem~B]{Muneta2008}, we get
\begin{align*}
\zeta(1,3;(4,1^{3}))
=
\zeta
\left(
\ 
\begin{ytableau}
 1 & *(gray)3 & 1 & *(gray)3 \\
 *(gray)3 \\
 1 \\
 *(gray)3 
\end{ytableau}
\ \right)
&=\frac{5}{64} \zeta(7) \zeta(4)^2-\frac{3}{32} \zeta(11) \zeta(4)+\frac{1}{64}\zeta(15),\\
\zeta(1,3;(3,1^{4}))
=
\zeta
\left(
\ 
\begin{ytableau}
 *(gray)3 & 1 & *(gray) 3 \\
 1 \\
 *(gray)3 \\
 1 \\
 *(gray)3 
\end{ytableau}
\ \right)
&=\frac{5}{896} \zeta(3) \zeta(4)^3-\frac{71}{896} \zeta(7) \zeta(4)^2+\frac{3}{32} \zeta(11) \zeta(4)-\frac{1}{64}\zeta(15).
\end{align*}
\end{example}
\subsection{Anti-hook types}
 An anti-hook is a skew Young diagram obtained by rotating a hook by $180^\circ$.
Anti-hook skew Young diagrams are of the form $(m^{n+1})/((m-1)^n)$ for some natural numbers $m\geq 2, n\geq 1$. Notice that they are always checkerboardable for any such $m,n$, since they just have one corner whose entry can be set to be $b$. We now study the checkerboard style Schur multiple zeta values of anti-hook types. 
 For example, 
\begin{equation*}
\ytableausetup{mathmode,boxsize=1.1em,aligntableaux=center} 
 \zeta(a,b;(3^5)/(2^4))
=
\zeta
\left(
\ 
\begin{ytableau}
 \none    & \none & *(gray)b \\
 \none    & \none & a \\
 \none    & \none & *(gray)b \\
 \none    & \none & a \\
 *(gray)b & a     & *(gray)b
\end{ytableau}
\ \right)
, \quad 
 \zeta(a,b;(4^3)/(3^2))
=
\zeta
\left(
\ 
\begin{ytableau}
 \none & \none    & \none & *(gray)b \\
 \none & \none    & \none & a \\
 a     & *(gray)b & a     & *(gray)b 
\end{ytableau}
\ \right)
\,.
\end{equation*}

We will now show, that the checkerboard style Schur multiple zeta values of anti-hook types
 can be written in terms of the ribbon types $A,B,S$ and $S^\star$, defined at the beginning of Section \ref{sec:primrib}.

 \begin{prop} \label{prop:antihook}
 \begin{enumerate}[i)]
 \item For $p\geq1$ and $q\geq 0$ we have  
 \begin{align*}
  \zeta(a,b;((2p)^{2q+1})/((2p-1)^{2q}))
 &=
 \sum_{\substack{1 \leq k_1 \leq p \\[2pt]1 \leq k_2 \leq q}} (-1)^{k_1+k_2}A_{a,b}(k_1+k_2-1)
 \zeta^{\star}\left(\{a,b\}^{p-k_1}\right)\zeta\left(b,\{a,b\}^{q-k_2}\right) \\
 &\ \ \ -\sum^{p}_{k_1=1}(-1)^{q+k_1} S_{a,b}^{\star}\left(q+k_1\right) \zeta^{\star}\left(\{a,b\}^{p-k_1}\right).
 \end{align*}

 \item For $p, q\geq1$ we have  
 \begin{align*}
  \zeta(a,b;((2p)^{2q})/((2p-1)^{2q-1}))
 &= 
  \sum_{\substack{1 \leq k_1 \leq p \\[2pt]1 \leq k_2 \leq q}} (-1)^{k_1+k_2}A_{a,b}(k_1+k_2-1)\,  \zeta^{\star}\left(\{a,b\}^{p-k_1}\right)\zeta\left(\{a,b\}^{q-k_2}\right)\,.
 \end{align*}
 \item For $p \geq 0$ and $q \geq 1$ we have    
 \begin{align*}
  \zeta(a,b;((2p+1)^{2q})/((2p)^{2q-1}))
 &=
 \sum_{\substack{1 \leq k_1 \leq p \\[2pt]1 \leq k_2 \leq q}} (-1)^{k_1+k_2} A_{a,b}(k_1+k_2-1)
 \zeta^{\star}\left(b,\{a,b\}^{p-k_1}\right)\zeta\left(\{a,b\}^{q-k_2}\right)\\
 &\ \ \ -\sum^{q}_{k_2=1}(-1)^{p+k_2}S_{a,b}\left(p+k_2\right) \zeta\left(\{a,b\}^{q-k_2}\right).
 \end{align*}
 \item For $p,q\geq 0$ we have   
 \begin{align*}
  \zeta(a,b;((2p+1)^{2q+1})/((2p)^{2q}))
 &=
 \sum_{\substack{1 \leq k_1 \leq p \\[2pt]1 \leq k_2 \leq q}} (-1)^{k_1+k_2} A_{a,b}(k_1+k_2-1)
 \zeta^{\star}\left(b,\{a,b\}^{p-k_1}\right)\zeta\left(b,\{a,b\}^{q-k_2}\right)\\
 &\ \ \ -\sum^{p}_{k_1=1}(-1)^{q+k_1} S_{a,b}^{\star}\left(q+k_1\right)\zeta^{\star}\left(b,\{a,b\}^{p-k_1}\right)\\
 &\ \ \ -\sum^{q}_{k_2=1}(-1)^{p+k_2}S_{a,b}\left(p+k_2\right)\zeta\left(b,\{a,b\}^{q-k_2}\right)\\
 &\ \ \ +(-1)^{p+q} B_{a,b}(p+q)\,.
 \end{align*}
 \end{enumerate}
 \end{prop} 
\begin{proof}
 These are directly obtained by using the harmonic product formula \eqref{for:harmonicproduct}. We demonstrate the proof for the case $\zeta(a,b;(3^5)/(2^4))$:
\begin{align*}
\ytableausetup{mathmode,boxsize=1.0em,aligntableaux=center} 
\zeta&
\left(
\ 
\begin{ytableau}
 \none    & \none & *(gray)b \\
 \none    & \none & a \\
 \none    & \none & *(gray)b \\
 \none    & \none & a \\
 *(gray)b & a     & *(gray)b
\end{ytableau}
\ \right)
=
\zeta\left(\ 
\begin{ytableau}
\none & \none & a \\
*(gray)b & a & *(gray)b 
\end{ytableau}
\ \right)
\cdot 
\zeta\left(\ 
\begin{ytableau}
*(gray)b \\
a \\
*(gray)b  
\end{ytableau}
\ \right)
-
\zeta\left(\ 
\begin{ytableau}
\none & \none & \none & *(gray)b \\
\none & \none & \none & a \\
\none & \none & a & *(gray)b \\
*(gray)b & a & *(gray)b 
\end{ytableau}
\ \right)
\\
&=
\zeta\left(\ 
\begin{ytableau}
\none & \none & a \\
*(gray)b & a & *(gray)b 
\end{ytableau}
\ \right)
\cdot 
\zeta\left(\ 
\begin{ytableau}
*(gray)b \\
a \\
*(gray)b  
\end{ytableau}
\ \right)
-
\zeta\left(\ 
\begin{ytableau}
\none & \none & \none & a \\
\none & \none & a & *(gray)b \\
*(gray)b & a & *(gray)b 
\end{ytableau}
\ \right)
\cdot 
\zeta\left(\ 
\begin{ytableau}
*(gray)b  
\end{ytableau}
\ \right)
+
\zeta\left(\ 
\begin{ytableau}
\none & \none & \none & a & *(gray)b \\
\none & \none & a & *(gray)b \\
*(gray)b & a & *(gray)b 
\end{ytableau}
\ \right)
\\
&=
\left[
\zeta\left(\ 
\begin{ytableau}
 \none & a \\
 a & *(gray)b 
\end{ytableau}
\ \right)
\cdot 
\zeta\left(\ 
\begin{ytableau}
*(gray)b  
\end{ytableau}
\ \right)
-
\zeta\left(\ 
\begin{ytableau}
 \none & a \\
 a & *(gray)b \\
 *(gray)b
\end{ytableau}
\ \right)
\right]
\zeta\left(\ 
\begin{ytableau}
*(gray)b \\
a \\
*(gray)b  
\end{ytableau}
\ \right)
-
\left[
\zeta\left(\ 
\begin{ytableau}
 \none & \none & a \\
 \none & a & *(gray)b \\
 a & *(gray)b 
\end{ytableau}
\ \right)
\cdot
\zeta\left(\ 
\begin{ytableau}
 *(gray)b 
\end{ytableau}
\ \right)
-
\zeta\left(\ 
\begin{ytableau}
 \none & \none & a \\
 \none & a & *(gray)b \\
 a & *(gray)b \\
 *(gray)b
\end{ytableau}
\ \right)
\right]
\zeta\left(\ 
\begin{ytableau}
*(gray)b  
\end{ytableau}
\ \right)
\\
&\ \ \ 
+
\left[
\zeta\left(\ 
\begin{ytableau}
 \none & \none & a & *(gray)b \\
 \none & a & *(gray)b \\
 a & *(gray)b 
\end{ytableau}
\ \right)
\cdot
\zeta\left(\ 
\begin{ytableau}
*(gray)b  
\end{ytableau}
\ \right)
-
\zeta\left(\ 
\begin{ytableau}
 \none & \none & a & *(gray)b \\
 \none & a & *(gray)b \\
 a & *(gray)b \\
 *(gray)b
\end{ytableau}
\ \right)
\right]
\\
&=A_{a,b}(1) \zeta^{\star}(b)\zeta(b,a,b) - A_{a,b}(2)\zeta^{\star}(b)\zeta(b) +S_{a,b}^{\star}(3)\zeta^{\star}(b)-S_{a,b}(2)\zeta(b,a,b)+S_{a,b}(3)\zeta(b)-B_{a,b}(3).
\end{align*}
\end{proof}

\begin{example} Using Theorem \ref{thm:SSstar} and \ref{thm:AB13} we give the following examples for $(a,b)=(1,3)$.
\begin{align*}
\zeta(1,3;(4^3)/(3^2))
=
\zeta
\left(
\ 
\begin{ytableau}
 \none & \none & \none & *(gray)3 \\
 \none & \none & \none & 1 \\
 1 & *(gray)3 & 1 & *(gray)3  
\end{ytableau}
\ \right)
&=\frac{5}{8} \zeta(3) \zeta(4) \zeta(5) -\frac{1}{8} \zeta(3) \zeta(9)-\frac{13245}{34496}\zeta(4)^3,\\
\zeta(1,3;(4^4)/(3^3))
=
\zeta
\left(
\ 
\begin{ytableau}
 \none & \none & \none & 1 \\
 \none & \none & \none & *(gray)3 \\
 \none & \none & \none & 1 \\
 1 & *(gray)3 & 1 & *(gray)3  
\end{ytableau}
\ \right)
&=\frac{5}{32} \zeta(4)^2 \zeta(5)-\frac{3}{16} \zeta(4) \zeta(9)+\frac{1}{32} \zeta(13), \\
\zeta(1,3;(3^4)/(2^3))
=
\zeta
\left(
\ 
\begin{ytableau}
 \none & \none & 1 \\
 \none & \none & *(gray)3 \\
 \none & \none & 1 \\
 *(gray)3 & 1 & *(gray)3  
\end{ytableau}
\ \right)
&=\frac{1}{8} \zeta(3) \zeta(4) \zeta(5) -\frac{1}{8} \zeta(3) \zeta(9) -\frac{493}{448448} \zeta(4)^3, \\
\zeta(1,3;(3^5)/(2^{4}))
=
\zeta
\left(
\ 
\begin{ytableau}
 \none & \none & *(gray)3 \\
 \none & \none & 1 \\
 \none & \none & *(gray)3 \\
 \none & \none & 1 \\
 *(gray)3 & 1 & *(gray)3  
\end{ytableau}
\ \right)
&=\frac{1}{8} \zeta(3)^2 \zeta(4) \zeta(5) +\frac{7279}{81536} \zeta(3) \zeta(4)^3 -\frac{1}{8} \zeta(3) \zeta(5) \zeta(7) \\
&\ \ \ +\frac{13}{896}\zeta(4)^2 \zeta(7) -\frac{1}{8} \zeta(3)^2 \zeta(9) -\frac{1}{64}\zeta(15).
\end{align*}
\end{example}
 
\begin{remark}\label{rem:integral}
Schur multiple zeta values of anti-hook types also appear in \cite{KanekoYamamoto}, where they are denoted for index sets ${\bf k}$, ${\bf l}$  by $\zeta({\bf k} \varoast {\bf l}^\star)$. For example with ${\bf k}=(b,a,b-1)$ and ${\bf l}=(a,b,a,1)$ we have $ \zeta(a,b;(4^3)/(3^2)) = \zeta({\bf k} \varoast {\bf l}^\star)$. Theorem 4.1. in \cite{KanekoYamamoto} gives an iterated integral expression for these $\zeta({\bf k} \varoast {\bf l}^\star)$, which was generalized in \cite{NakasujiPhuksuwanYamasaki} for arbitrary ribbons. These iterated integral expressions could be used to obtain even more formulas for Schur multiple zeta values of ribbon type.
\end{remark}
 
\begin{proof}[Proof of Theorem~\ref{thm:13rib}]
Again by the harmonic product it is easy to see, that every Schur multiple zeta values of Checkerboard style ribbon can be expressed as a linear combination of products of the primitive ribbons $A$,$B$,$S$, $S^\star$ and the multiple zeta values $\zeta(b,\{a,b\}^n)$ and $\zeta(\{a,b\}^n)$. In the case $(a,b) = (1,3)$ we know by Theorem \ref{thm:AB13} that $A_{1,3}(n)= \frac{2}{4^n} \zeta(4n+1)$ for $n\geq 1$ and $B_{1,3}(n) = \frac{1}{4^n} \zeta(4n+3)$ for $n\geq 0$. Together with $\zeta(\{1,3\}^n) = \frac{1}{4^n} \zeta(\{4\}^n) \in \Q \pi^{4n}$ and Corollary \ref{cor:XXsYYs} we see that the $\Q$-vector space spanned by all Schur multiple zeta values of Checkerboard style ribbons in theses cases is exactly  $\Q[\pi^4,\zeta(3),\zeta(5),\zeta(7),\dots]$. 
\end{proof}

\section{Stairs}
Recall that for a natural number $N \in \N$ we write $\delta_N=(N,N-1,\ldots,2,1)$. We call a Young diagram of the form $\delta_N\slash\mu$ with $\mu \subset \delta_{N}$ a stair. One example for a stair is the primitive ribbon $B_{a,b}(n)$ discussed in the section before. In this section we calculate Checkerboard style Schur multiple zeta values of stair type and prove that these are all given as determinants of matrices whose entries are given by $B_{a,b}(n)$. Before we can state the precise result, we need to introduce some notations.
 
 For $N \in \N$ and a partition $\mu=(\mu_1,\ldots,\mu_N)\subset \delta_{N}$, where we also allow $\mu_i=0$,
 define
\begin{align*}
 J_{0}(\mu)=J_{N,0}(\mu)&=\{j\in \{1,2,\ldots,N\}\,|\,N+j \not\equiv \mu_j \!\!\! \pmod{2}\} \,, \\
 J_{1}(\mu)=J_{N,1}(\mu)&=\{j\in \{1,2,\ldots,N\}\,|\,N+j \equiv \mu_j \!\!\! \pmod{2}\} \,.
\end{align*}
 Notice that $J_{0}(\mu)\cap J_{1}(\mu)=\emptyset$ and $J_{0}(\mu)\sqcup J_{1}(\mu)=\{1,2,\ldots,N\}$.
 Moreover, for $j\in\{1,2,\ldots,N\}$, put
\begin{align*}
 m_{j}(\mu)=m_{N,j}(\mu)
  =
  \begin{cases}
   \frac{N+j-1-\mu_j}{2} & j\in J_{0}(\mu) \,, \\
   \frac{N+j-2-\mu_j}{2} & j\in J_{1}(\mu) \,,
  \end{cases}
 \end{align*} 
and set 
\[
 l_N(\mu)=\sum_{j\in J_{0}(\mu)}(m_{j}(\mu)+j+1).
\]
 Furthermore, define the $N \times N$-matrix $B_N(\mu)=B_{N,a,b}(\mu)$ by 
\[
 B_N(\mu)=\left(\,(-1)^{m_{j}(\mu) - i  + 1}B(m_{j}(\mu) - i  + 1) \,\right)_{1 \leq i,j \leq N},
\]
 where $B(n)=B_{a,b}(n)$ is the primitive ribbon defined at the beginning of Section \ref{sec:primrib}. By $F_N(\mu)=F_{N,a,b}(\mu)$ we denote the submatrix of $B_N(\mu)$ of size $|J_{1}(\mu)|\times |J_{1}(\mu)|$,
 obtained by removing  from $B_N(\mu)$ for all $j\in J_{0}(\mu)$ the $j$-th column and the $(m_{j}(\mu)+1)$-th row.
\begin{thm}
\label{thm:main}
 For $N\in\N$ and a partition $\mu=(\mu_1,\ldots,\mu_N)\subset \delta_{N}$ we have
\begin{align}
\label{for:stair}
 \zeta(a,b;\delta_N/\mu)
&=(-1)^{l_N(\mu)}\det\left( F_N(\mu)\right)\,,\\
\label{for:stairtranspose}
 \zeta(a,b;\delta_N/\mu)
&=(-1)^{l_N(\mu')}\det\left( F_N(\mu') \right)\,.
\end{align}
\end{thm}

 We first show some examples.
  
\begin{example}
  Let us consider the case of $N=5$. 
\begin{enumerate}[i)]
\item
 When $\mu=(2,2,1)$,
 we have $J_{0}(\mu)=\{2,3,4\}$, $J_{1}(\mu)=\{1,5\}$ and hence $(m_{1}(\mu),\ldots,m_{5}(\mu))=(1,2,3,4,4)$.
 This yields $l_5(\mu)=21$,
\[
 B_5(\mu)
 \, = \,
\begin{pmatrix}
 -B(1) & {\color{darkgray}B(2)} & {\color{darkgray}-B(3)} & {\color{darkgray}B(4)} & B(4) \\
 B(0) & {\color{darkgray}-B(1)} & {\color{darkgray}B(2)} & {\color{darkgray}-B(3)} & -B(3) \\
 {\color{darkgray}0} & {\color{darkgray}B(0)} & {\color{darkgray}-B(1)} & {\color{darkgray}B(2)} & {\color{darkgray}B(2)}  \\
 {\color{darkgray}0} & {\color{darkgray}0} & {\color{darkgray}B(0)} & {\color{darkgray}-B(1)} & {\color{darkgray}-B(1)} \\
 {\color{darkgray}0} & {\color{darkgray}0} & {\color{darkgray}0} & {\color{darkgray}B(0)} & {\color{darkgray}B(0)} \\ 
\end{pmatrix}\  
 \text{ and } \  
 F_5(\mu)
 \, = \,
\begin{pmatrix}
 -B(1) & B(4) \\
 B(0) & -B(3)
\end{pmatrix}
 \,.
\]
 Hence, from \eqref{for:stair},
\[
\ytableausetup{mathmode,boxsize=1.2em,aligntableaux=center} 
\zeta\left(
\ {\footnotesize
\begin{ytableau}
 \none & \none & *(gray) b & a & *(gray) b \\
 \none & \none & a & *(gray) b \\
 \none & a & *(gray) b \\
 a & *(gray) b \\
 *(gray) b
\end{ytableau}}
\
\right)
\,=\,(-1)^{21}
\begin{vmatrix}
 -B(1) & B(4) \\
 B(0) & -B(3)
\end{vmatrix}
\,=\,-
\begin{vmatrix}
 B(1) & B(4) \\
 B(0) & B(3)
\end{vmatrix}
 \,.
\]
 \item
 When $\mu=(2,2)$,
 we have $J_{0}(\mu)=\{2,4\}$, $J_{1}(\mu)=\{1,3,5\}$ and hence $(m_{1}(\mu),\ldots,m_{5}(\mu))=(1,2,3,4,4)$.
 This yields $l_5(\mu)=14$,
\[
 B_5(\mu)
 \, = \,
\begin{pmatrix}
 -B(1) & {\color{darkgray}B(2)}& -B(3) & {\color{darkgray}B(4)}& B(4) \\
 B(0) & {\color{darkgray}-B(1)}& B(2) &{\color{darkgray}-B(3)} & -B(3) \\
 {\color{darkgray}0} & {\color{darkgray}B(0)} & {\color{darkgray}-B(1)} &{\color{darkgray}B(2)} &{\color{darkgray}B(2)} \\
 0 & {\color{darkgray}0} & B(0) & {\color{darkgray}-B(1)}& -B(1) \\
 {\color{darkgray}0} & {\color{darkgray}0} & {\color{darkgray}0} & {\color{darkgray}B(0)} & {\color{darkgray}B(0)} \\ 
\end{pmatrix}\ 
 \text{ and } \ 
 F_5(\mu) 
 \, = \,
\begin{pmatrix}
 -B(1) & -B(3) & B(4) \\
 B(0) & B(2) & -B(3) \\
 0 & B(0) & -B(1)
\end{pmatrix}
\,.
\]
 Hence, from \eqref{for:stair},
\[
\ytableausetup{mathmode,boxsize=1.2em,aligntableaux=center} 
\zeta\left(
\ {\footnotesize
\begin{ytableau}
 \none & \none & *(gray) b & a & *(gray) b \\
 \none & \none & a & *(gray) b \\
 *(gray) b & a & *(gray) b \\
 a & *(gray) b \\
 *(gray) b
\end{ytableau}}
\
\right)
\,=\,
 (-1)^{14}
\begin{vmatrix}
 -B(1) & -B(3) & B(4) \\
 B(0) & B(2) & -B(3) \\
 0 & B(0) & -B(1)
\end{vmatrix}
\,=\,
\begin{vmatrix}
 B(1) & B(3) & B(4) \\
 B(0) & B(2) & B(3) \\
 0 & B(0) & B(1)
\end{vmatrix}
\,.
\] 
\end{enumerate} 
\end{example}

 From \eqref{for:Brec}, we have 
\begin{equation}
\label{for:1F0}
\zeta(b,\{a,b\}^n)
=\sum^{n}_{k=0}(-1)^{n-k}B(n-k) \zeta(\{a,b\}^k)\,. 
\end{equation}
 Let $e_1,\ldots,e_N\in\mathbb{R}^N$ be the standard basis of $\mathbb{R}^N$.
 For $n\ge 0$, define $t(n)=t_N(n)\in\R^N$, $y(n)=y_N(n)\in\R^N$ and $b(n)=b_N(n)\in\R^N$ by
\[
 t(n)=\sum^{N}_{i=1} \zeta(\{a,b\}^{n+1-i}) e_i\,, \quad   
 y(n)=\sum^{N}_{i=1} \zeta(b,\{a,b\}^{n+1-i}) e_i\,, \quad 
 b(n)=\sum^{N}_{i=1}(-1)^{n+1-i}B(n+1-i)e_i\,.
\]
 Notice that, from \eqref{for:1F0}, we have 
\begin{equation}
\label{for:1F0vector}
 y(n)=\sum^{n}_{i=0}(-1)^{n-i}B(n-i)t(i)\,. 
\end{equation}

 The following is elementary but useful in our discussion.
 
\begin{lemma}
\label{lem:elementarytranceformation}
 Let $n_1,\ldots,n_k,m$ be positive integers satisfying $n_1<\cdots<n_k\leq m<N$.
 Then, by the elementary transformation,
 the matrix $\left(t(n_1),\ldots,t(n_k),y(m) \right)$ of size $N\times (k+1)$ is reduced to 
\[
 \left( e_{n_1+1},\ldots,e_{n_k+1},b(m)\right)\,.
\] 
\end{lemma}
\begin{proof}
 For matrices $P,Q$, 
 let us write $P\to Q$ if $Q$ is obtained by some elementary transforms from $P$.
 
 When $k=1$, noticing $T(0)=1$ and using \eqref{for:1F0vector}, we have
\begin{align*}
 \,&\left( t(n_1),\sum^{m}_{i=0}(-1)^{m-i}B(m-i)t(i) \right)\\
\to\,&\left( t(n_1),\sum^{n_1}_{i=0}(-1)^{m-i}B(m-i)t(i)+\sum^{m+1}_{i=n_1+2}(-1)^{m+1-i}B(m+1-i)e_i \right)\\
\to\,&\left( t(n_1),\sum^{n_1-1}_{i=0}(-1)^{m-i}B(m-i)t(i)+\sum^{m+1}_{i=n_1+2}(-1)^{m+1-i}B(m+1-i)e_i \right)\\
\to\,&\left( e_{n_1+1},\sum^{n_1-1}_{i=0}(-1)^{m-i}B(m-i)t(i)+\sum^{m+1}_{i=n_1+2}(-1)^{m+1-i}B(m+1-i)e_i \right)\\
\to\,&\left( e_{n_1+1},\sum^{n_1}_{i=1}(-1)^{m+1-i}B(m+1-i)e_i+\sum^{m+1}_{i=n_1+2}(-1)^{m+1-i}B(m+1-i)e_i \right)\\
=\,&\left( e_{n_1+1},b(m)-(-1)^{m-n_1}B(m-n_1)e_{n_1+1} \right)\to\,\left( e_{n_1+1},b(m) \right).
\end{align*}
 Hence the claim follows.
 The general cases are verified by induction on $k$.
\end{proof}

 Now we give a proof of Theorem~\ref{thm:main}.
\begin{proof}
[Proof of Theorem~\ref{thm:main}]
 We first show the equation \eqref{for:stairtranspose}.
 From the Jacobi-Trudi formula \eqref{for:JTzeta},
 we have $\zeta(a,b;\delta_N/\mu^\prime)=\det( \alpha_1,\ldots,\alpha_{N} )$
 where $\alpha_j\in\R^N$ is defined by  
\[
 \alpha_j
=
\begin{cases}
 t(m_j(\mu)) & j\in J_0(\mu), \\
 y(m_j(\mu)) & j\in J_1(\mu). 
\end{cases}
\]
 Moreover, from Lemma~\ref{lem:elementarytranceformation},
 we have $\zeta(a,b;\delta_N/\mu^\prime)=\det( \beta_1,\ldots,\beta_N )$
 where $\beta_j\in\R^N$ is defined by   
\[
 \beta_j
=
\begin{cases}
 e_{m_j(\mu)+1} & j\in J_0(\mu), \\
 b(m_j(\mu)) & j\in J_1(\mu). 
\end{cases}
\]
 Now, it is easy to see that this is equal to $(-1)^{l_N(\mu)}\det( F_N(\mu) )$ and therefore interchanging $\mu$ with $\mu^\prime$ we obtain $\zeta(a,b;\delta_N/\mu)=(-1)^{l_N(\mu')}\det\left( F_N(\mu') \right)$.

 The equation \eqref{for:stair} is similarly obtained by the same discussion above. For this we use the Jacobi-Trudi formula \eqref{for:JTzetastar} for $\zeta(a,b;\delta_N/\mu)$ instead of  \eqref{for:JTzeta} for $\zeta(a,b;\delta_N/\mu^\prime)$, which will give the exact same matrix as before, except that $\zeta$ will be replaced by $\zeta^\star$. Since equation \eqref{for:1F0} is also true for $\zeta^\star$ the statement in Lemma \ref{lem:elementarytranceformation} remains true by replacing $t$ by $t^\star $ and $y$ by $y^\star$, which are defined by
  \[
   t^\star(n)=\sum^{N}_{i=1} \zeta^\star(\{a,b\}^{n+1-i}) e_i\,, \quad   
   y^\star(n)=\sum^{N}_{i=1} \zeta^\star(b,\{a,b\}^{n+1-i}) e_i\,.
  \]
  Therefore we also obtain $\zeta(a,b;\delta_N/\mu) = \det( \beta_1,\ldots,\beta_N ) = (-1)^{l_N(\mu)}\det( F_N(\mu) )$.
\end{proof}

 As a special case, we have the following.
 
\begin{thm}
\label{thm:deltaNdeltan}
 For $N\in\N$ and $0\le n\le N-1$, 
 $\zeta(a,b;\delta_N\slash\delta_n)$ can be written as a Hankel determinant of the matrices in $B(n)$.
 More precisely, 
\begin{enumerate}[i)]
\item
 When $N\equiv n$ \!\!\! $\pmod{2}$, we have
\[
 \zeta(a,b;\delta_N/\delta_n)
=\det\left( B(i+j+n-1) \right)_{1\leq i,j\leq \frac{N-n}{2}}\,.
\]
\item
 When $N\not\equiv n$ \!\!\! $\pmod{2}$, we have 
\[
 \zeta(a,b;\delta_N/\delta_n)
=(-1)^{\frac{1}{2}n(n+1)}\det\left( B(i+j-n-2) \right)_{1\leq i,j\leq \frac{N+n+1}{2}}\,.
\]
\end{enumerate}
\end{thm}
\begin{proof}
 This follows from Theorem~\ref{thm:main}
 with the following data which can be obtained by direct calculations.
\begin{enumerate}[i)]
\item When $N\equiv n$ $\pmod{2}$, letting $m=\frac{N-n}{2}$, we have
\begin{align*}
 J_0(\delta_n)
&=\{1,2,\ldots,n\}\sqcup\{n+1,n+3,\ldots,N-1\}\,, \\
 J_1(\delta_n)
&=\{n+2,n+4,\ldots,N\}\,, \\ 
 \{m_j(\delta_n)\,|\,j\in J_0(\delta_n)\}
&=\left\{\frac{N-n}{2},\frac{N-n+2}{2},\ldots,\ldots,N-1\right\}=\left\{m+j-1\,\left|\,1\leq j\leq\frac{N+n}{2}\right.\right\}\, \\
 \{m_j(\delta_n)\,|\,j\in J_1(\delta_n)\}
&=\left\{\frac{N+n}{2},\frac{N+n+2}{2},\ldots,N-1\right\}\,, \\
 l_N(\delta_n)
&\equiv mn+\frac{1}{2}m\left( m+1 \right) \!\!\! \pmod{2}\,, \\  
 F_N(\delta_n)
&=\left( (-1)^{m+1-i+j+n-1}B\left( m+1-i+j+n-1 \right) \right)_{1\leq i,j\leq m}\,.
\end{align*}
 Hence, we have 
\begin{align}
 (-1)^{l_N(\delta_n)}&\det\left( F_N(\delta_n) \right) \nonumber \\
\label{for:Nneven}
&=(-1)^{mn+\frac{1}{2}m(m+1)}\cdot (-1)^{\frac{1}{2}m(m-1)+m(n-1)+m(m+1)}\det\left( B(i+j+n-1) \right)_{1\leq i,j\leq m}\\
&=\det\left( B(i+j+n-1) \right)_{1\leq i,j\leq m}\,. \nonumber
\end{align}
\item When $N\not\equiv n$ $\pmod{2}$, letting $m=\frac{N+n+1}{2}$, we have
\begin{align*}
 J_0(\delta_n)
&=\{n+2,n+4,\ldots,N-1\}=\left\{n+2j\,\left|\,1\leq j\leq\frac{N-n-1}{2}\right.\right\}\,, \\
 J_1(\delta_n)
&=\{1,2,\ldots,n\}\sqcup\{n+1,n+3,\ldots,N\}\,, \\
 \{m_j(\delta_n)\,|\,j\in J_0(\delta_n)\}
&=\left\{\frac{N+n+1}{2},\frac{N+n+3}{2},\ldots,N-1\right\}=\left\{m+j-1\,\left|\,1\leq j\leq\frac{N-n-1}{2}\right.\right\}\,, \\
 \{m_j(\delta_n)\,|\,j\in J_1(\delta_n)\}
&=\left\{\frac{N-n-1}{2},\frac{N-n+1}{2},\ldots,N-1\right\}\,, \\
 l_N(\delta_n)&\equiv \frac{1}{2}(m-n-1)\left(3m+n \right) \!\!\! \pmod{2}\,,\\
 F_N(\delta_n)
&=\left( (-1)^{m+1-i+j-n-2}B\left( m+1-i+j-n-2 \right) \right)_{1\leq i,j\leq m}\,.
\end{align*}
 Hence, we have 
\begin{align}
 (-1)^{l_N(\delta_n)}&\det\left( F_N(\delta_n) \right) \nonumber \\
\label{for:Nnodd}
&=(-1)^{\frac{1}{2}(m-n-1)\left(3m+n \right)}\cdot (-1)^{\frac{1}{2}m(m-1)-m(n+2)}\det\left( B(i+j-n-2) \right)_{1\leq i,j\leq\frac{N+n+1}{2}}\\
&=(-1)^{\frac{1}{2}n(n+1)}\det\left( B(i+j-n-2) \right)_{1\leq i,j\leq\frac{N+n+1}{2}}\,. \nonumber
\end{align}
\end{enumerate} 
 Notice that, in the equalities \eqref{for:Nneven} and \eqref{for:Nnodd},
 we have used the formulas
\begin{align*}
 \det( a_{m+1-i,j} )_{1\leq i,j\leq m}
&=(-1)^{\frac{1}{2}m(m-1)}\det(a_{i,j})_{1\leq i,j\leq m}\,,\\
 \det( c^{i+j}a_{i,j} )_{1\leq i,j\leq m}
&=c^{m(m+1)}\det(a_{i,j})_{1\leq i,j\leq m}\,.
\end{align*}

\end{proof}

\begin{example}
 When $N=5$, we have 
\begin{align*}
\ytableausetup{mathmode,boxsize=1.2em,aligntableaux=center} 
\zeta\left(
\ {\footnotesize
\begin{ytableau}
 *(gray) b & a & *(gray) b & a & *(gray) b \\
 a & *(gray) b & a & *(gray) b \\
 *(gray) b & a & *(gray) b \\
 a & *(gray) b \\
 *(gray) b
\end{ytableau}}
\
\right)
\,&=\,
\begin{vmatrix}
 B(0) & B(1) & B(2) \\
 B(1) & B(2) & B(3) \\
 B(2) & B(3) & B(4) 
\end{vmatrix}
, &   
\zeta\left(
\ {\footnotesize
\begin{ytableau}
 \none & a & *(gray) b & a & *(gray) b \\
 a & *(gray) b & a & *(gray) b \\
 *(gray) b & a & *(gray) b \\
 a & *(gray) b \\
 *(gray) b
\end{ytableau}}
\
\right)
\,&=\,
\begin{vmatrix}
 B(2) & B(3) \\
 B(3) & B(4) 
\end{vmatrix}
,\\[5pt]
\zeta\left(
\ {\footnotesize
\begin{ytableau}
 \none & \none & *(gray) b & a & *(gray) b \\
 \none & *(gray) b & a & *(gray) b \\
 *(gray) b & a & *(gray) b \\
 a & *(gray) b \\
 *(gray) b
\end{ytableau}}
\
\right)
\,&=\,- 
\begin{vmatrix}
 0 & 0 & B(0) & B(1) \\
 0 & B(0) & B(1) & B(2) \\
 B(0) & B(1) & B(2) & B(3) \\
 B(1) & B(2) & B(3) & B(4) 
\end{vmatrix}
, & 
\zeta\left(
\ {\footnotesize
\begin{ytableau}
 \none & \none & \none & a & *(gray) b \\
 \none & \none & a & *(gray) b \\
 \none & a & *(gray) b \\
 a & *(gray) b \\
 *(gray) b
\end{ytableau}}
\
\right)
\,&=\,
\begin{vmatrix}
 B(4) 
\end{vmatrix}
\,.
\end{align*} 
\end{example}

 When $(a,b)=(1,3)$, from Theorem~\ref{thm:deltaNdeltan} and \eqref{for:B13}, we have the following formulas.
 
\begin{cor}\label{cor:31case}
 For $N\in\N$ and $0\le n\le N-1$, 
 $\zeta(1,3;\delta_N/\delta_n)$ can be written as a Hankel determinant of the matrices in $\zeta(4k+3)$.
 More precisely, 
\begin{enumerate}[i)]
\item
 When $N\equiv n$ \!\!\! $\pmod{2}$,
 we have
\begin{align*}
 \zeta(1,3;\delta_N/\delta_n)
&=4^{-\frac{1}{4}(N+n)(N-n)}
\det\left( \zeta\left(4(i+j+n)-1\right) \right)_{1\le i,j\le \frac{N-n}{2}}.
\end{align*}
\item
 When $N\not\equiv n$ \!\!\! $\pmod{2}$,
 we have 
\begin{align*}
 \zeta(1,3;\delta_N/\delta_n)
&=(-1)^{\frac{1}{2}n(n+1)}4^{-\frac{1}{4}(N+n+1)(N-n-1)}
\det\left( \zeta\left(4(i+j-n)-5\right)\right)_{1\le i,j\le \frac{N+n+1}{2}},
\end{align*}
 where we set $\zeta(n)=0$ if $n<0$.
\end{enumerate}
 In particular,  
 $\zeta(1,3;\delta_N/\delta_n)\in\mathbb{Q}[\zeta(4n+3)\,|\,n\ge 0]$.
\end{cor}

\begin{example}
 When $N=5$, we have 
\begin{align*}
\ytableausetup{mathmode,boxsize=1.2em,aligntableaux=center} 
\zeta\left(
\ 
{\footnotesize\begin{ytableau}
 *(gray) 3 & 1 & *(gray) 3 & 1 & *(gray) 3 \\
 1 & *(gray) 3 & 1 & *(gray) 3 \\
 *(gray) 3 & 1 & *(gray) 3 \\
 1 & *(gray) 3 \\
 *(gray) 3
\end{ytableau}}
\
\right)
\,&=\,
\frac{1}{4^6}
\begin{vmatrix}
 \zeta(3) & \zeta(7) & \zeta(11) \\
 \zeta(7) & \zeta(11) & \zeta(15) \\
 \zeta(11) & \zeta(15) & \zeta(19) 
\end{vmatrix}
, & 
\zeta\left(
\ 
{\footnotesize\begin{ytableau}
 \none & 1 & *(gray) 3 & 1 & *(gray) 3 \\
 1 & *(gray) 3 & 1 & *(gray) 3 \\
 *(gray) 3 & 1 & *(gray) 3 \\
 1 & *(gray) 3 \\
 *(gray) 3
\end{ytableau}}
\
\right)
\,&=\,
\frac{1}{4^6}
\begin{vmatrix}
 \zeta(11) & \zeta(15) \\
 \zeta(15) & \zeta(19) 
\end{vmatrix}
,\\[5pt]
\zeta\left(
\ {\footnotesize
\begin{ytableau}
 \none & \none & *(gray) 3 & 1 & *(gray) 3 \\
 \none & *(gray) 3 & 1 & *(gray) 3 \\
 *(gray) 3 & 1 & *(gray) 3 \\
 1 & *(gray) 3 \\
 *(gray) 3
\end{ytableau}}
\
\right)
\,&=\,
-\frac{1}{4^4}
\begin{vmatrix}
 0 & 0 & \zeta(3) & \zeta(7) \\
 0 & \zeta(3) & \zeta(7) & \zeta(11) \\
 \zeta(3) & \zeta(7) & \zeta(11) & \zeta(15) \\
 \zeta(7) & \zeta(11) & \zeta(15) & \zeta(19) 
\end{vmatrix}
, & 
\zeta\left(
\ {\footnotesize
\begin{ytableau}
 \none & \none & \none & 1 & *(gray) 3 \\
 \none & \none & 1 & *(gray) 3 \\
 \none & 1 & *(gray) 3 \\
 1 & *(gray) 3 \\
 *(gray) 3
\end{ytableau}}
\
\right)
\,&=\,
\frac{1}{4^4}
\begin{vmatrix}
 \zeta(19) 
\end{vmatrix}
 \,.
\end{align*} 
\end{example}

\section{Other shapes, observations and discussion}

We end this note by discussing the other shapes which were not discussed above and also some numerical observations and possible further directions.\\

{\bf Squares:} Of course speaking about Checkerboard style Schur multiple zeta values, the square shapes comes into mind. Though the authors could not find any nice formulas for these shapes. One can check that 
\begin{align}\label{eq:2x2square}
\zeta\left(\ {\footnotesize
				\begin{ytableau}
				*(gray)3 & 1 \\
				1 & *(gray)3 
				\end{ytableau}}\ \right) 
				&	= \frac{1}{2} \zeta(3) \zeta(5) - \frac{5}{16} \zeta(4)^2\,,
\end{align}
but the $3 \times 3$ case seems not to be a polynomial in single zeta values anymore. To deal with the square case one possibility is again to use the Jacobi-Trudi formula to get for example for $a,b \geq 2$
\begin{align*}
\zeta\left(\ {\footnotesize
				\begin{ytableau}
				*(gray)b & a & *(gray)b\\
				a & *(gray)b & a \\
				*(gray)b & a & *(gray)b
				\end{ytableau}}\ \right) 
				&	=  \zeta(b,a,b,a,b) \zeta\left(\ {\footnotesize
								\begin{ytableau}
								*(gray)b & a \\
								a & *(gray)b 
								\end{ytableau}}\ \right) - \zeta(b,a,b,a) \zeta\left(\ {\footnotesize
																\begin{ytableau}
																*(gray)b & a & *(gray)b \\
																a & *(gray)b 
																\end{ytableau}}\ \right) + \zeta(b,a,b) \zeta\left(\ {\footnotesize
																																\begin{ytableau}
				*(gray)b & a & *(gray)b \\
				a & *(gray)b & a\end{ytableau}}\ \right)\,.
\end{align*}
This works for the arbitrary $n \times n$-case and therefore it seems to be necessary to investigate the Checkerboard style Schur multiple zeta values of shape $\lambda = (n,\dots,n,n-1,\dots,n-1)$. \\

{\bf Anti-stairs:} We saw that $\zeta\left(\ {\footnotesize \begin{ytableau}
	\none & 1  \\
	1 & *(gray)3
\end{ytableau}}\ \right) 
= \frac{1}{2} 
\zeta(5) $ and one can also show that
\begin{align*}
	\zeta\left(\ {\footnotesize
					\begin{ytableau}
					\none & \none & *(gray)3 \\
					\none & *(gray)3 & 1  \\
					*(gray)3 & 1  & *(gray)3 
					\end{ytableau}}\ \right) 
										&= \frac{1}{16} \zeta(5) \zeta(9) -  \frac{1}{16} \zeta(3)^2 \zeta(8) + \frac{1}{2} \zeta(3)^3 \zeta(5) + \frac{3}{8} \zeta(3) \zeta(4) \zeta(7) - \frac{85}{192} \zeta(5)^2 \zeta(4)\,.
\end{align*}
This can be done by using a regularized version of the Jacobi-Trudi formulas. But for higher weight cases it is not clear how to deal with these "anti-stair" type of Schur multiple zeta values.\\

{\bf Gluings:} Equation \eqref{eq:2x2square} can also be written as 
\[ 
\zeta\left(\ {\footnotesize \begin{ytableau}
	*(gray)3 
\end{ytableau}}\ \right) \cdot 
\zeta\left(\ {\footnotesize \begin{ytableau}
	\none & 1  \\
	1 & *(gray)3
\end{ytableau}}\ \right)  
- \zeta\left(\ {\footnotesize	\begin{ytableau}
	*(gray)3 & 1 \\
	1 & *(gray)3 
\end{ytableau}}\ \right) 
 = 70\, \zeta(\{1,3\}^2) \,.\]
Taking the product of these types of Schur multiple zeta values and substract their "glued"-version seems to give multiple of $\zeta(\{1,3\}^n)$ also in higher weight cases. From numerical experiments we expect the following
\begin{align}
\zeta\left(\ {\footnotesize \begin{ytableau}
	1 & *(gray)3\\
	*(gray)3
\end{ytableau}}\ \right) \cdot 
\zeta\left(\ {\footnotesize \begin{ytableau}
\none & \none & 1 \\
\none & 1 & *(gray)3 \\
	1 & *(gray)3 
\end{ytableau}}\ \right)  
- \zeta\left(\ {\footnotesize	\begin{ytableau}
1 & *(gray)3 & 1 \\
*(gray)3 & 1 & *(gray)3 \\
	1 & *(gray)3 
\end{ytableau}}\ \right)&\overset{?}{=} 1074502\, \zeta(\{1,3\}^4)\,, \\ 
\zeta\left(\ {\footnotesize \begin{ytableau}
\none & 1 & *(gray)3\\
1 & *(gray)3 \\
*(gray)3
\end{ytableau}}\ \right) \cdot 
\zeta\left(\ {\footnotesize \begin{ytableau}
\none & \none & \none & 1\\
\none & \none & 1 & *(gray)3\\
\none & 1 & *(gray)3\\
1 & *(gray)3
\end{ytableau}}\ \right)  
- \zeta\left(\ {\footnotesize	\begin{ytableau}
\none & 1 &  *(gray)3 & 1\\
1 &  *(gray)3 & 1 & *(gray)3\\
 *(gray)3 & 1 & *(gray)3\\
1 & *(gray)3
\end{ytableau}}\ \right)&\overset{?}{=} \frac{9656199193420}{21}\, \zeta(\{1,3\}^6)\,,
\end{align}
and similarly for the weight $32$ case, in which the product subtracted by the "glued"-version is expected to be 
$2222659435447178310\,\zeta(\{1,3\}^8)$. So one might expect that for $n \geq 1$ there exists $\alpha_n \in \Q$ with $B_{1,3}(n-1) \cdot A_{1,3}(n)  - \zeta(1,3; (n+1,n+1,n,n-1,\dots,3,2)/\delta_{n-2}) \overset{?}{=} \alpha_n \zeta(\{1,3\}^{2n})$.\\

{\bf Interpolated versions:} Another possible direction is to study formulas for Checkerboard style interpolated Schur multiple zeta values. Interpolated Schur multiple zeta values were introduced in \cite{Bachmann}. They are elements $\zeta^t(\kk) \in \R[t]$, which interpolate between a Schur multiple zeta value of a Young tableau $\kk$ and its conjugate $\kk'$, i.e. $\zeta^0(\kk) = \zeta(\kk)$ and $\zeta^1(\kk) =  \zeta(\kk')$.
In the case $\lambda=(2,1)$ we have for $a\geq 1, b,c \geq 2$
\begin{align*}	\zeta^t \left(\ {\footnotesize \begin{ytableau}
	a & b  \\
	c
	\end{ytableau}}\ \right) &= \zeta(a,b,c) +\zeta(a,c,b) + \zeta(a+b,c)+\zeta(a,b+c) \\
	&+ \big( \zeta(a+c,b) - \zeta(a+b,c) + \zeta(a+b+c) \big) \cdot t - \zeta(a+b+c) \cdot  t^2  \,.
\end{align*}
As an analogue of $B_{1,3}(1)=\zeta\left(\ {\footnotesize \begin{ytableau}
	1 & *(gray)3  \\
	*(gray)3
	\end{ytableau}}\ \right) 
		= \frac{1}{4} 
			\zeta(7) $, which is a special case of Corollary \ref{cor:31case},
we get
\[ \zeta^t\left(\ {\footnotesize \begin{ytableau}
	1 & *(gray)3  \\
	*(gray)3
	\end{ytableau}}\ \right) 
		= \left( \frac{1}{4} + t (1-t) \right)  \zeta(7)\,.\]
 Furthermore we get as an analogue of $A_{1,3}(2)=\frac{1}{8}\zeta(9)$ 
\[\zeta^t\left(\ {\footnotesize \begin{ytableau}
\none & \none & 1 \\
\none & 1 & *(gray)3\\
1 & *(gray)3
\end{ytableau}}\ \right)  = \frac{1}{8} \zeta(9) +\left(\frac{15}{2} \zeta(5) \zeta(4) - \frac{31}{4} \zeta(9)\right) t (1-t) + \zeta(9) t^2 (1-t)^2 \,.\]
Since in \cite{Bachmann} a Jacobi-Trudi formula for $\zeta^t(\kk)$ is proven,  one could try to give an analogue of Corollary \ref{cor:31case} for the interpolated case. To give a formula for the general case one would need explicit evaluations of the interpolated multiple zeta values $\zeta^t(\{1,3\}^n)$ and $\zeta^t(3,\{1,3\}^n)$.\\

{\bf Other special cases for $a$ and $b$:}  A natural question is whether there are other $a,b$, such that the $A_{a,b}(n)$ or $B_{a,b}(n)$ have a nice explicit form. In the case $(a,b) = (1,2)$ it is well-known that due to duality we have $\zeta(\{1,2\}^n) = \zeta(\{3\}^n)$. From this one can also deduce, together with Lemma \ref{lem:ABrec} and \ref{lem:SSstargenseries}, that for all $n\geq 1$
\begin{align*}
\label{for:A13}
\ytableausetup{mathmode,boxsize=1.2em,aligntableaux=center} 
 A_{1,2}(n)
=\zeta
\left(
\ {\footnotesize
\begin{ytableau}
 \none & \none    & \none  & 1        \\
 \none & \none    & \adots & *(gray)2 \\
 \none & 1        & \adots \\
 1     & *(gray)2  
\end{ytableau}}
\ \right)
=3 \zeta(3n+1)\,\quad\text{ and }\quad
S_{1,2}(n)=\zeta\left(
\ {\footnotesize
\begin{ytableau}
 \none  & \none & 1 \\
 \none  & \adots& *(gray) 2 \\
 1 & \adots \\
 *(gray) 2
\end{ytableau}}
\
\right) = \zeta^\star(\{3\}^n) \,.
 \end{align*} 
Though in contrast to the $(a,b)=(1,3)$ case, it seems that $B_{1,2}(n)$ is not a rational multiple of $\zeta(3n+2)$.

In the case $a=b=2k$ with $k\geq 1$, it is clear that all $A_{2k,2k}(n), B_{2k,2k}(n)$ and $S_{2k,2k}(n)$ are some rational multiples of even powers of $\pi$. 
For other cases with $a\neq b$ there are no explicit nice evaluations available for the primitive ribbon cases. In the case $(a,b)=(1,5)$ there are partial results to evaluate $\zeta(\{1,5\}^n))$ in \cite{Yee}.\\ 

{\bf Other proofs:} It would be interesting to know if the  simple explicit sum representation of the odd single zeta values in Theorem \ref{thm:AB13} also have a more elementary proof. 
In \cite{Monien} and \cite{HaynesZudilin2015} the authors also studied Hankel determinants of single zeta values. It would be interesting to know if the explicit calculation of the Hankel determinant, as it was done in \cite{Monien}, can also be used to prove our identities in Theorem \ref{cor:31case} directly without using the Jacobi-Trudi formulas. 

\section*{Acknowledgement}
The authors would like to thank Don Zagier for his ideas on parts of the proof of Theorem \ref{thm:AB13}, Wadim Zudilin for helpful discussion on the topic and the Max-Planck-Institut f\"ur Mathematik in Bonn for hospitality and support.


\noindent
\textsc{Henrik Bachmann}\\
Graduate School of Mathematics, \\
 Nagoya University, Nagoya, Japan \\
 \texttt{henrik.bachmann@math.nagoya-u.ac.jp }\\
 
\noindent
\textsc{Yoshinori Yamasaki}\\
 Graduate School of Science and Engineering, \\
 Ehime University, Ehime, Japan \\
 \texttt{yamasaki@math.sci.ehime-u.ac.jp}

\end{document}